\documentclass[leqno]{amsart}
\usepackage{amsmath, amsthm, amssymb, amstext}

\usepackage{hyperref,xcolor}
\hypersetup{
 pdfborder={0 0 0},
 colorlinks,
}
\usepackage[left=2.5cm,right=2.5cm,top=3cm,bottom=3cm]{geometry}

\usepackage[colorinlistoftodos]{todonotes}
\usepackage{enumitem}
\allowdisplaybreaks

\newtheorem{theorem}{Theorem}[section]
\newtheorem{lemma}[theorem]{Lemma}
\newtheorem{e-proposition}[theorem]{Proposition}
\newtheorem{corollary}[theorem]{Corollary}

\newtheorem{e-definition}[theorem]{Definition}
\newtheorem{remark}[theorem]{Remark}
\newtheorem{example}[theorem]{Example}

\def\hh{\mathcal H^{n-1}}
\DeclareMathOperator*{\divergenz}{div}              %
\DeclareMathOperator*{\sgn}{sgn}            %
\DeclareMathOperator*{\esssup}{ess ~sup}         %

\newcommand{\N}{\mathbb{N}}

\newcommand{\R}{\mathbb{R}}

\newcommand{\RN}{\mathbb{R}^n}

\newcommand{\dd}{\, d}

\newcommand{\Omt}{\Omega_{t}}

\newcommand{\intot}{\int_{\Omega_{t}}}
\newcommand{\intort}{\int_{E_{t}}}

\newcommand{\HN}{\mathcal{H}^{n-1}}

\newcommand{\eps}{\varepsilon}

\newcommand{\Om}{\Omega}
\newcommand{\rand}{\partial\Omega}

\newcommand{\into}{\int_{\Omega}}
\newcommand{\intor}{\int_{\partial\Omega}}

\renewcommand{\l}{\left}
\renewcommand{\r}{\right}
\numberwithin{theorem}{section}
\numberwithin{equation}{section}

\newcommand{\pina}[1]{\todo[inline,color=yellow!40]{Pina: #1}}

\newcommand{\M}{{\mathcal M}}

\newcommand{\ch}{\mathcal{H}}

\title[Boundedness of solutions to elliptic boundary value problems in Orlicz spaces]{Boundedness of solutions to Dirichlet, Neumann and Robin problems for elliptic equations in Orlicz spaces}

\author[G.\,Barletta]{Giuseppina Barletta}
\address[G.\,Barletta]{Dipartimento di Ingegneria Civile, dell'Energia, dell'Ambiente e dei Materiali, Universit\`a Mediterranea di Reggio Calabria, Via Graziella- Loc. Feo di Vito, 89122  Reggio Calabria, Italy}
\email{giuseppina.barletta@unirc.it}

\author[A.\,Cianchi]{Andrea Cianchi}
\address[A.\,Cianchi]{Dipartimento di Matematica e Informatica U. Dini, Universit\`a di Firenze, Viale Morgagni 67/A, 50134 Firenze, Italy}
\email{andrea.cianchi@unifi.it}

\author[G.\,Marino]{Greta Marino}
\address[G.\,Marino]{Technische Universit\"{a}t Chemnitz, Fakult\"at f\"ur Mathematik, Reichenhainer Stra\ss e 41, 09126 Chemnitz, Germany}
\email{greta.marino@mathematik.tu-chemnitz.de}



\begin{document}

\begin{abstract}
Boundary value problems for second-order  elliptic equations in divergence form, whose nonlinearity is  governed by a convex function of non-necessarily power type, are considered. The global boundedness of their solutions is established under boundary conditions of  Dirichlet, or Neumann, or Robin type. A decisive role in the results   is played by  optimal   forms of Orlicz-Sobolev embeddings and boundary trace embeddings, which allow for critical growths  of the coefficients.
\end{abstract}

\maketitle


%
\section{Introduction}

We are concerned with the global boundedness of weak solutions to  boundary value problems for second-order nonlinear elliptic equations of the form
	\begin{equation}\label{eq}
	-\divergenz \big(\mathcal A(x, u, \nabla u)\big)= \mathcal B(x, u, \nabla u) \quad \text{in } \Om, 
%
	\end{equation}
where  $\Om$ is an open set  in $\RN$, $n\geq 2$, with finite Lebesgue measure $|\Om|$. 

 This primary question in the regularity theory  
 was extensively analyzed as early as in the the classical monograph \cite{LadUr}, where the boundedness of solutions to elliptic equations was also shown to be crucial for enjoying finer regularity properties.
 Specifically, in  \cite[Theorem 7.1, Chapter 4, and Theorem 3.1, Chapter 5]{LadUr} this problem for equations as in  \eqref{eq} is  discussed in the standard case when Dirichlet boundary conditions are imposed  and the nonlinearities of the functions $\mathcal A$ and $\mathcal B$ in $u$ and $\nabla u$ are of polynomial type. This amounts to saying that  the functions $\mathcal A$ and $\mathcal B$ are subject to growth conditions  depending  on powers of $u$ and $|\nabla u|$.  Over the years, the results of \cite{LadUr} have been sharpened, and extended to different boundary conditions in several contributions, such as \cite{GP, GV, MMM, MWinkert1, MWinkert2, MWinkert3, Wang, Winkert}.

%

In the present paper, the restriction 
on power type  growths is removed, and the boundedness of  solutions to boundary value problems for equations of type \eqref{eq} driven by more general nonlinearities is established. A distinctive trait of our analysis is that the ellipticity condition on the principal part is dictated by a  general nonnegative convex function vanishing   at $0$ -- a Young function -- of the modulus of the gradient. This calls for replacing  customary Sobolev spaces with Orlicz-Sobolev spaces as a functional framework for solutions. 

More specifically, we assume that $\mathcal A\colon \Om \times \R \times \RN \to \RN$ and $\mathcal B\colon \Om \times \R \times \RN \to \R$ are Carath\'eodory functions obeying:
\begin{equation}\label{hpA}
	\mathcal A(x, t, \xi) \cdot \xi \ge A(|\xi|)- B(|t|) 
\end{equation}
and
\begin{equation}\label{hpB}
\sgn(t) \mathcal B(x, t, \xi) \le f(x)+ C(|t|)+ D(|t|) E(|\xi|)
\end{equation}
 for a.e. $x \in \Om$, for every $t \in \R$ and every $\xi \in \RN$.
Here, the dot $\lq\lq \cdot "$ stands for the scalar product in $\RN$, $A: [0, \infty) \to [0, \infty)$ is a Young function, $B, C, D, E : [0, \infty) \to [0, \infty)$ are increasing functions,
and $f: \Omega \to [0, \infty)$ is a measurable function whose integrability degree is prescribed by the membership in a weak Orlicz space defined by a Young function $M$.
%


Dirichlet, Neumann and Robin boundary value problems are included in our discussion. The datum $u_0$ in the Dirichlet boundary condition
\begin{equation*}
u=u_0 \quad \text{on $\partial \Omega$}
\end{equation*}
is naturally assumed to be bounded. The weak formulation of the corresponding problem allows for arbitrary  domains $\Omega$ such that  $|\Omega|<\infty$. 

The Neumann condition to be imposed amounts to the vanishing of the co-normal derivative associated with the principal part of  equation \eqref{eq}; namely:
\begin{equation*}
\mathcal A(x, u, \nabla u) \cdot \nu = 0 \quad \text{on $\partial \Omega$,}
\end{equation*}
where $\nu$ denotes the normal unit vector on $\partial\Omega$. The notion of weak solution to the corresponding Neumann problem avoids any a priori regularity assumption on $\Omega$. However, its degree of (ir)regularity enters the assumptions required on the functions appearing on the right-hand sides of inequalities \eqref{hpA} and \eqref{hpB}. The regularity of $\Omega$ will be prescribed in terms of its membership in Maz'ya type classes of bounded domains,   defined via relative isoperimetric inequalities.

When a Robin condition
\begin{equation}\label{rob} 
\mathcal A(x, u, \nabla u) \cdot \nu = \mathcal C(x, u) \quad \text{on $\partial \Omega$}
\end{equation}
is assigned, boundary traces of solutions are as well involved  in the definition of weak solution to the problem. The existence of traces is only guaranteed provided that $\Omega$ is regular enough. Minimal assumptions on $\Omega$ will be imposed in this connection.
The function $\mathcal C\colon \rand \times \R  \to \R$ in equation \eqref{rob} is Carath\'eodory  and is assumed to fulfill the bound
\begin{equation}\label{hpC}
\sgn(t) \mathcal C(x, t) \le g(x)+  F(|t|)  
\end{equation}
 for $\mathcal H^{n-1}$-- a.e. $x \in \partial \Om$ and for every $t \in \R$. Here, $\mathcal H^{n-1}$ denotes the $(n-1)$-dimensional measure,  $F: [0, \infty) \to [0, \infty)$ is an increasing function, and the function $g : \partial \Om \to \mathbb R$ belongs to some weak Orlicz space, with respect to the measure $\mathcal H^{n-1}$, built upon a Young function $N$.


Our result on Dirichlet problems amounts to balance conditions among the growths  of  functions $A$,$B$,$C$,$D$,$E$,\\
$M$
for their    solutions to be bounded. An analogous result holds for Neumann problems, save that the   conditions also depend on an additional parameter, which appears as an exponent in the relative isoperimetric inequality supported by the domain $\Omega$. A conjugate of the function $A$  from an optimal Orlicz-Sobolev inequality comes into  the picture of these conditions.  In fact, growth conditions in $u$ which are critical in connection with the Sobolev conjugate of $A$ are permitted.
As for Robin problems, the same assumptions as in the case of Dirichlet problems are required, plus additional balances between the functions $A$ and $F$, and between the functions $A$ and $N$. In this connection,  a conjugate of $A$ from a sharp Orlicz-Sobolev boundary trace inequality emerges as well. A critical growth in $u$ according  to the Sobolev trace conjugate of $A$ is also entitled. 

In the situation when the functions $A, B, C, D, E, F$ appearing in inequalities \eqref{hpA}, \eqref{hpB} and \eqref{hpC} are mere powers, the theorems to be presented recover several instances of those appearing in the contributions cited above, and also generalize them  to the case of Neumann problems in possibly irregular domains. 

The analysis of the global boundedness of solutions to elliptic problems under Orlicz type growth and Dirichlet boundary conditions traces back to  \cite{Talenti79} in the case when $B=C=D=E=0$, and to \cite{Talenti90} in a complementary setting when $f=0$  and $D=E=0$. The results of the latter paper were  augmented in \cite{Cianchi-boundedness}, where  a sharp growth condition on the functions $B$ and $C$ was exhibited for  solutions to the Dirichlet problems to be bounded. An extension to anisotropic equations is the subject of \cite{alberico}. 

These  are the contributions, dealing with elliptic equations ruled by Orlicz type nonlinearities, which are most closely related to the topic of our paper.  On the other hand, the study of diverse  
regularity properties of solutions to elliptic equations and systems with Orlicz growth   has seen an increasing interest in the last decades. They are nowadays the subject of a vast literature, a nonexaustive sample of which includes  
\cite{AbMo, ADF, BaSu, BC, Baroni, BeckMingione,  BHHK, BrMo, BreitSV,  BGKS, CaKP, IC-pocket, CeOk, CGZ, CGSW,  CianchiAIHP,  DeFMi, DieningSV, Gossez2, HHT, Kor, Li, Mar, MuTi, Wo}. The monograph \cite{CGSW} is   an up-to-date source of results and references in this area.

 As far as  proofs of our results are concerned, we point out that standard methods, such as Moser type iteration techniques, do not seem to yield optimal conclusions.  This is related to the fact that, unlike the case when power type  bounds in \eqref{hpA}--\eqref{hpC} are imposed, iterating optimal inequalities in the Orlicz framework need not produce an inequality which is again optimal. This drawback is also  visible, for instance, in the iteration of Orlicz-Sobolev embeddings -- see e.g. \cite{CPS} for a discussion on this point.

Our approach   makes use of identities and inequalities involving integrals over the  level sets of the solutions.  Starting with the work of De Giorgi, this is of course nowadays also customary  in the regularity theory of  PDEs. However, we have to depart from usual discretization techniques applied to the levels of integration, which apparently do not  properly fit general growth conditions. As a replacement, we derive   a differential inequality involving the distribution function of the solution, which can 
 be regarded as a continuous version of these discretization techniques, and
is
reminiscent of  arguments from \cite{Cianchi-boundedness, Talenti79, Talenti90}.  

More specifically, we start with  bounds for  integrals involving $\nabla u$ by integrals involving $u$ over level sets of the latter. These bounds are, in a sense, global Caccioppoli type inequalities, and are deduced from assumptions \eqref{hpA}--\eqref{hpC} by testing the weak formulation of the problems under consideration by truncations of the solutions themselves. The resultant inequalities are coupled with inequalities in the opposite direction for integrals of $u$ in terms of integrals of  $\nabla u$. 

The main  novelties are to be found  in this key step, 
which strongly depends on the use of  various   sharp embeddings and respective inequalities in Orlicz-Sobolev spaces. Not only inequalities with an optimal Orlicz target space, but also improvements with a target of Orlicz-Lorentz type, which is in fact optimal among all rearrangement-invariant   spaces, come into play. In dealing with Robin problems, Orlicz-Sobolev trace inequalities  have equally a pivotal role. In addition,   a new  Sobolev trace inequality in the so called Lorentz $\Lambda$-spaces, depending on the Young function $N$, is needed.

The combination of these upper and lower bounds for energy integrals of the solutions provides us with an inequality which, via  the coarea formula and isoperimetric inequalities applied to the level sets of the solutions, yields a differential inequality involving their distribution function. The regularity of the domain $\Omega$ required for Neumann and Robin problems enters the game at this stage. 
An analysis of  the solutions to the relevant differential inequality   finally dictates the  conditions to be imposed on the functions appearing in    assumptions \eqref{hpA}--\eqref{hpB}, and, in the case of Robin problems,  also in  \eqref{hpC}.

\section{Orlicz-Sobolev spaces}\label{spaces}

The function space background underlying the proofs of our main results is recalled in this section. Additional material is presented in Section \ref{traceineq}, where an optimal trace inequality in Lorentz $\Lambda$-spaces, taylored for our applications to Robin problems, is established.

\subsection{Young functions}

A function $A \colon [0, \infty) \to [0, \infty] $ is said to be a Young function if it is convex, left-continuous, vanishes at $0$, and is neither identically equal to zero, nor to infinity. Any Young function $A$ takes the form
\begin{equation*}
A(t) = \int_0^ta(\tau)\, d\tau \quad \text{for $t \geq0$,}
\end{equation*}
where the function $a: [0, \infty) \to [0, \infty]$ is  non-decreasing and left-continuous. 
 If $A$ is a Young function, then the function 
\begin{equation}\label{mono}
\frac{A(t)}t \quad \text{is non-decreasing in $(0, \infty)$.}
\end{equation}
Moreover,
\begin{equation}\label{Ak}
A(kt) \geq kA(t) \quad \text{for $k\geq 1$ and $t\geq 0$.}
\end{equation}
The Young conjugate of $A$ is denoted by $\widetilde A$, and  is defined as
	\[
	\widetilde A(s)= \sup\{st- A(t): \, t \ge 0\} \quad \text{for } s \ge 0.
	\]
One has that $\widetilde{\widetilde{A}}= A$. By the very definition of $\widetilde A$,  
	\begin{equation}
	\label{young}
	ts \le A(t)+ \widetilde A(s) \quad \text{for $s,t \geq 0$.}
	\end{equation}
Moreover, one has that
	\begin{equation}
	\label{young'}
	t \le \widetilde A^{-1}(t) A^{-1}(t) \le 2t  \quad \text{for $t \geq 0$,}
	\end{equation}
where the superscript \lq\lq $-1$" stands for the (generalized) right-continuous inverse.
In particular, inequality \eqref{young'} yields
	\begin{equation}
	    \label{dis-A-tildeA}
	\frac{A(t)}{t} \le \widetilde A^{-1}(A(t)) \le 2 \frac{A(t)}{t} \quad \text{for  } t>0.
	\end{equation}
The function $A$ is said to satisfy the $\Delta_2$-condition near infinity (briefly, $A \in \Delta_2$ near infinity) if $A$ is finite-valued and there exist constants $c \ge 2$ and $t_0 \ge 0$ such that
	\begin{equation}
	\label{delta2}
	A(2t) \le c A(t) \quad \text{for } t \ge t_0.
\end{equation}
The function $A$  is said to satisfy the $\nabla_2$-condition near infinity (briefly, $A \in \nabla_2$ near infinity) if there exist constants $c> 2$ and $t_0 \ge 0$ such that
	\begin{equation}
	\label{nabla2}
	A(2t) \ge c A(t) \quad \text{for } t \ge t_0.
	\end{equation}
If inequality \eqref{delta2}, or \eqref{nabla2}, holds with $t_0=0$, then $A$ is said to satisfy the $\Delta_2$-condition globally, or the $\nabla_2$-condition globally, respectively. 
\\ One has that
\begin{equation}\label{deltanabla}
\text{$A \in \nabla _2$ near infinity  [globally] if and only if $\widetilde A \in \Delta _2$  near infinity [globally].}
\end{equation}
A Young function $A_1$ is said to dominate another Young function $A_2$ near infinity  [globally] if there exist constants $c>0$ and $t_0\geq 0$ such that
\begin{equation*} 
A_2(t) \leq A_1(ct) \quad \text{for $t \geq t_0$ \,[$t \geq 0$].} 
\end{equation*}
If the functions $A_1$ and $A_2$ dominate each other near infinity [globally],  then they are said to be equivalent near infinity [globally]. The relations of domination and of equivalence between the functions $A_1$ and $A_2$ will be denoted by $A_2 \lesssim A_1$ and $A_2 \approx A_1$, respectively.

\subsection{Orlicz and weak Orlicz spaces}
Let $(\mathcal R, m)$ be a non-atomic, $\sigma$-finite measure space, equipped with a measure $m$. In our applications, $\mathcal R$ will be either a subset of $\RN$, endowed with the Lebesgue measure,  or a subset of $\partial \Omega$, endowed with the Hausdorff measure $\mathcal H^{n-1}$.
We denote by $\mathcal M(\mathcal R)$ the space of real-valued $m$-measurable functions on $\mathcal R$.
\\
Given a Young function $A$, 
the Orlicz class $K^A(\mathcal R)$ is defined as 
\begin{equation*}
K^A(\mathcal R) = \bigg\{u\in \M (\mathcal R): \,\, \int_\mathcal R A(|u|)\,dm <\infty\bigg\}.
\end{equation*}
The set $K^A(\mathcal R)$ is, in general, just a convex subset of $\M (\mathcal R)$.  One the other hand, the Orlicz space $L^A(\mathcal R)$, defined as
\begin{equation*}
L^A(\mathcal R) = \big\{u\in \M (\mathcal R):\,\, u/\lambda \in K^A(\mathcal R) \,\, \text{for some} \,\,\lambda \in \R\big\},
\end{equation*}
is a Banach space, equipped with the norm
 \begin{equation*}
    \|u\|_{L^{A}(\mathcal R)} = \inf \l\{\lambda>0 :\, \into A\l(\frac{|u|}{\lambda}\r) dm \le 1 \r\}. 
        \end{equation*}
As a consequence of inequality \eqref{young}, one has that
\begin{equation}\label{june21}
\int_{\mathcal R} |uv|\, dm \leq \int_{\mathcal R} A(|u|)\, dm +  \int_{\mathcal R} \widetilde A(|v|)\, dm
\end{equation}
for $u\in K^A(\mathcal R)$ and $v \in K^{\widetilde A}(\mathcal R)$.
\\
Assume that $m(\mathcal R)<\infty$. Given two Young functions $A_1$ and $A_2$, one has that
\begin{equation*}
L^{A_1}(\mathcal R) \to L^{A_2}(\mathcal R) 
\end{equation*}
if and only if $A_1$ dominates $A_2$ near infinity. Here, and in what follows, the arrow \lq\lq $\to$" stands for continuous embedding. 
\\ Given a  function $u\in \M (\mathcal R)$, we denote by $u^* : [0, \infty) \to [0, \infty]$ its decreasing rearrangement, defined as 
\begin{equation*}
u^*(s) = \inf\{t\geq 0: m(\{|u|>t\})\leq s\} \quad \text{for $s \geq 0$.}
\end{equation*}
The Hardy-Littlewood inequality tells us that
\begin{equation}\label{HL}
\int _{\mathcal R} |uv|\, dm \leq \int_0^\infty u^*(s) v^*(s)\, ds
\end{equation}
for $u, v \in \mathcal M(\mathcal R)$.
\\ The function $u^{**} : (0, \infty) \to [0, \infty]$ is defined as
    \begin{equation*}
	u^{**}(s)= \frac{1}{s} \int_0^s u^*(r) \;dr \quad \text{for }  s>0.
	\end{equation*}
Note that $u^{**}$ is also non-increasing, and 
\begin{equation*} 
u^*(s) \leq u^{**}(s) \quad \text{ for $s>0$.}
\end{equation*}
The signed decreasing rearrangement $u^\circ : [0, \infty) \to [-\infty, \infty]$ is given by 
\begin{equation*}
u^\circ(s) = \inf\{t\geq 0: m(\{u>t\})\leq s\} \quad \text{for $s\geq 0$.}
\end{equation*}
The functions $u$, $u^*$ and $u^\circ$ are equimeasurable.
\\ When $m(\mathcal R)<\infty$, the increasing rearrangement $u_*: [0, m(\mathcal R)] \to [0, \infty]$ is defined as
$$u_*(s) = u^*(m(\mathcal R) -s)  \quad \text{for $s \in [0, m(\mathcal R)]$.}$$
 A reverse Hardy-Littlewood inequality reads
\begin{equation}\label{HLrev}
\int _{\mathcal R} |uv|\, dm \geq \int_0^{m(\mathcal R)} u^*(s) v_*(s)\, ds
\end{equation}
for $u, v \in \mathcal M(\mathcal R)$.
\\
Under the assumption that $m(\mathcal R)<\infty$, we also set 
\begin{equation*}
{\rm med } (u) = u^\circ \big(\tfrac {m(\mathcal R)}2\big),
\end{equation*}
a median of $u$.
\\
 Given a Young function $A$, the weak Orlicz space $L^{A,\infty}(\mathcal R)$ is defined as the set of functions $u\in \M (\mathcal R)$ for which the norm
\begin{equation*}
\|u\|_{L^{A, \infty}(\mathcal R)} = \sup_{s>0} \frac{u^{**}(s)}{A^{-1}(1/s)}
\end{equation*}
is finite. The space $L^{A,\infty}(\mathcal R)$ is a Banach space endowed with this norm. Its name stems from the fact that $L^A(\mathcal R) \to L^{A, \infty}(\mathcal R)$ for every Young function $A$, and 
$$  \|u\|_{L^{A, \infty}(\mathcal R)} \leq \|u\|_{L^{A}(\mathcal R)}
$$
 for $u \in L^{A}(\mathcal R)$. Weak Orlicz spaces are a special instance of the Marcinkiewicz spaces, whose definition will be recalled in Section \ref{traceineq}.
\\ A variant of inequality \eqref{june21}, of critical use in our proofs, tells us what follows. 
Let $\alpha >1$ and let $A$ be a Young function 
such that 
\begin{equation}\label{convint0alpha}
\int_0\l(\frac{t}{A(t)}\r)^{\frac{1}{\alpha -1}} dt < \infty
\end{equation}
and 
\begin{equation}\label{intdivalpha}
\int^\infty \l(\frac{t}{A(t)}\r)^{\frac{1}{\alpha -1}} dt = \infty.
\end{equation}
Let $\widehat A_\alpha $ be the Young function defined by
\begin{equation}\label{hatAn}
\widehat A_\alpha (t) = \int_0^t \widehat a_\alpha (\tau)\, d\tau \quad \text{for $t \geq 0$,}
\end{equation}
and 
\begin{equation}\label{hatan}
\widehat a_\alpha ^{-1}(s) = \bigg(\int _{a^{-1}(s )}^{\infty}\bigg(\int _0^t
\bigg(\frac{1}{a(\tau)}\bigg)^{\frac{1}{\alpha -1}}
 d\tau\bigg)^{-\alpha}\frac{dt}{a(t )^{\frac{\alpha}{\alpha-1
}}}\bigg)^{\frac{1}{1-\alpha }}\,\,\,\quad{\rm
for}\,\,\,
 s \geq 0\,.
\end{equation}
Then there exists a constant  $\kappa_1=\kappa_1(\alpha)$
 such that
\begin{equation}\label{youngcianchi}
	\int_{\mathcal R} |uv| \,dm \le \kappa_1 \int_0^{\infty} \widehat A_\alpha \l(s^{-1/\alpha} u^*(s)\r) ds+ \kappa_1 \int_0^{\infty} \widetilde A\l(s^{-1/\alpha'} \int_0^s v^*(r) \,dr \r) ds
\end{equation}
for every $u, v \in \mathcal M(\mathcal R)$, see
 \cite[eq. (4.46)]{Cianchi-2004}. Note that the convergence of the inner integral on the right-hand side of equation \eqref{hatan} is   equivalent to assumption \eqref{convint0alpha}. 

\subsection{Orlicz-Sobolev spaces}

 Let $\Om$ be an open set  in $\RN$, and let $A$ be a Young function.  We define the Orlicz-Sobolev class $V^1K^A(\Om)$ as
\begin{equation*}
	V^1K^A(\Om)= \l\{u \in \M(\mathcal R): \,\,   u \text{ is weakly differentiable in $\Om$ and $|\nabla u| \in K^A(\Om)$}\r\}.
	\end{equation*}
The subset of $V^1K^A(\Om)$ of those functions which vanish, in a suitable sense, on $\partial \Omega$ can be defined as 
	\begin{align*}
	V_0^1K^A(\Om)= \bigl\{u \in \M(\mathcal R) & :\text{the continuation of $u$ by $0$ outside of $\Om$} \\ 
	\nonumber & \qquad \text{ is weakly differentiable in $\RN$ and belongs to $V^1K^A(\RN)$}\bigr\}.
	\end{align*}
	The Orlicz-Sobolev space  $W^{1, A}(\Om)$,  and  its subspace $W^{1, A}_0(\Om)$, are accordingly defined as 
    \begin{equation*}
        W^{1, A}(\Om) = \l\{u \in L^A(\Om): u/\lambda  \in  V^1K^A(\Om)  \,\, \text{for some $\lambda \in \R$} \r\},
    \end{equation*}
and 
    \begin{equation*}
            W^{1,A}_0(\Om) =  \l\{u  \in L^A(\Om): u/\lambda  \in  V^1_0K^A(\Om)  \,\, \text{for some $\lambda \in \R$} \r\}.
    \end{equation*}
The space $W^{1,A}(\Om)$, equipped with the norm
  \begin{equation}\label{norm}
    \|u\|_{W^{1,A}(\Om)}= \|u\|_{L^A(\Om)}+ \|\nabla u\|_{L^A(\Om)},
    \end{equation}
is a Banach space. Here $\|\nabla u\|_{L^A(\Om)}$ stands for $\||\nabla u|\|_{L^A(\Om)}$. The subspace $W^{1,A}_0(\Om)$ is  closed in $W^{1,A}(\Om)$ with respect to this  norm, and hence is also a Banach space. A Poincar\'e type inequality ensures that, if $|\Omega|<\infty$, then the functional $\|\nabla u\|_{L^A(\Om)}$ defines a norm on $W^{1,A}(\Om)$ equivalent to the norm \eqref{norm}.

\subsection{Orlicz-Sobolev embeddings}

Let $\Omega$ be an open set in $\RN$ such that $|\Omega|<\infty$. Given a Young function $A$, let $H_n  \colon [0, \infty) \to [0, \infty) $ be the function defined as
	\begin{equation}
	\label{H}
	H_n(t)= \l(\int_0^t \l(\frac{\tau}{A(\tau)}\r)^{\frac{1}{n-1}} d\tau \r)^{\frac{1}{n'}} \quad \text{for } t \ge 0,
	\end{equation}
where $n'= \frac n{n-1}$ and $A$ is modified, if necessary, near $0$ in such a way that
\begin{equation}\label{intconv0}
\int_0\l(\frac{t}{A(t)}\r)^{\frac{1}{n-1}} \,dt < \infty.
\end{equation}
The  Sobolev conjugate  of $A$ is the Young function  $A_n$, introduced in \cite{Cianchi-boundedness}, and  given by
	\begin{equation}
	\label{An}
	A_n(t)= A(H_n^{-1}(t)) \quad \text{for } t \ge 0.
	\end{equation}
 Here, $H_n^{-1}$ denotes the (generalized) left-continuous inverse of $H_n$.  One has that
\begin{equation}\label{embw0}
    W_0^{1,A}(\Om) \to L^{A_n}(\Om),
   \end{equation}
and $L^{A_n}(\Om)$ is the optimal (i.e. smallest possible) Orlicz target space \cite{Cianchi-boundedness}, see also  \cite{Cianchi-embedding} for an alternate equivalent formulation. In particular, there exists a positive constant $\kappa_2=\kappa_2(n)$ such that
    \begin{equation}
       \label{eq:poincare}
        \into A_{n} \l(\frac{\kappa_2 |u|}{\l(\into A(|\nabla u|) \,dy\r)^{1/n}} \r) dx \le \into A(|\nabla u|) \,dx
    \end{equation}
for every $u \in V^1_0K^A(\Omega)$.
\\
Notice that,  if $A$ grows so fast near infinity that
\begin{equation}\label{intconv}
\int ^\infty \l(\frac{t}{A(t)}\r)^{\frac{1}{n-1}} dt < \infty,
\end{equation}
then 
\begin{equation}\label{embinf}
W_0^{1,A}(\Om) \to L^\infty(\Om).
\end{equation}
Actually, condition \eqref{intconv} implies
that $H_n^{-1}(t)=\infty$ for large $t$. Hence, $A_n(t) = \infty$ as well, and $L^{A_n}(\Om)= L^\infty (\Omega)$, up to equivalent norms. 
\\ On the other hand,  the space $W_0^{1,A}(\Om)$ does contain unbounded functions in the opposite regime when
\begin{equation}\label{intdiv}
\int ^\infty \l(\frac{t}{A(t)}\r)^{\frac{1}{n-1}} dt = \infty.
\end{equation}
If assumption \eqref{intdiv}  is in force, embedding \eqref{embw0}  can still be augmented,
provided that the class of admissible target  spaces is enlarged as to include all rearrangement-invariant spaces. The optimal target
for embeddings of $W^{1, A}_0(\Om)$ into spaces from  this
 class is an Orlicz-Lorentz space, equipped with the norm given by
\begin{equation*}
\|u\|_{L(\widehat A_n, n)(\Omega)} = \big\|s^{-\frac 1n}
u^* (s)\big\|_{L^{\widehat A_n}(0, |\Omega|)}
\end{equation*}
for $u \in \M (\Omega)$, where $\widehat A_n$ is the Young function defined as in \eqref{hatAn}--\eqref{hatan}, with $\alpha =n$.
%
Note that, as in 
 the case of  definition \eqref{H} for the function $A_n$, the function $A$  is modified near zero, if necessary, in such a way that assumption  \eqref{intconv0} is satisfied. Such an assumption is equivalent to requiring that the innermost integral   on the right-hand side of equation \eqref{hatan}, with $\alpha =n$, converges.
\\
The relevant embedding  reads
 \begin{equation}
    \label{embol}    
   W_0^{1, A}(\Om) \to L(\widehat A_n, n)(\Omega).
        \end{equation}
Moreover, a positive constant $\kappa_3=\kappa_3(\Omega)$ exists such that
\begin{equation}\label{intol0}
\int_0^{|\Omega|}\widehat A_\alpha \big(\kappa_3 s^{- \frac 1n} u^*(s)\big)\, ds \leq \int_\Omega  A(|\nabla u|)\, dx
\end{equation}
for every $u \in V^1_0K^A(\Omega)$.
%
%
%
%
%
%
%
%

The optimal Orlicz target space for embeddings of the Orlicz-Sobolev space $W^{1,A}(\Om)$ is the same as in \eqref{embw0}, provided that   $\Omega$ is sufficiently regular. Bounded Lipschitz  domains, bounded domains satisfying the cone property, bounded John domains, and   $(\varepsilon, \delta)$-domains are admissible, for instance.  Any domain from these classes supports  a relative isoperimetric inequality of the form
\begin{equation}
    \label{isop}
    \kappa_4 \min \{|E|, |\Om \setminus E| \}^{ \frac 1{\alpha'}} \le \mathcal H^{n-1}(\partial ^M E \cap \Om)
    \end{equation}
with $\alpha =n$, for some positive constant $\kappa_4>0$ and every measurable set $E \subset \Om$. Here,  $\partial ^M$ stands for essential boundary in the sense of geometric measure theory, and $ \mathcal H^{n-1}(\partial ^M E \cap \Om)$ is the perimeter of $E$, relative to $\Om$.
\\ The name of inequality \eqref{isop} stems from the fact that it is a version in an open set $\Omega$ of the classical isoperimetric inequality in $\RN$
\begin{equation}\label{isoprn}
    n \omega_n^{\frac 1n}  \min \{|E|, |\RN \setminus E| \}^{ \frac 1{n'}} \le \mathcal H^{n-1}(\partial ^M E)
  \end{equation}
for every  measurable set $E\subset \RN$, where $\omega_n$ denotes the Lebesgue measure of the unit ball in $\RN$.
\\
The embedding 
$$W^{1,A}(\Om) \to L^{A_n}(\Om),$$
for bounded domains satisfying inequality \eqref{isop} with $\alpha =n$, is established in  \cite[Theorem 6.12]{CPS} and in \cite{Cianchi-embedding}, with $A_n$ replaced with an equivalent Young function.
\\ In the same papers, Sobolev embeddings for the space $W^{1,A}(\Om)$ are also offered when $\Omega$ is a  bounded open set
fulfilling inequality \eqref{isop} for some 
$\ \alpha > n$. The use of relative isoperimetric inequalities in the characterization of Sobolev embedding was discovered by V. Maz'ya more than sixty years ago \cite{Ma1} (see also \cite{Mabook} for an expanded treatment of this topic). We shall denote by $\mathcal{J}_{1/{\alpha'}}$ the Maz'ya class of open sets satisfying inequality \eqref{isop}.
\\ The Orlicz target space in Orlicz-Sobolev embedding on these domains is defined in terms of the Young function $A_{\alpha}$ given by
    \begin{equation*}
        A_{\alpha}(t)= A(H_{\alpha}^{-1}(t)) \quad \text{for } t \ge 0,
    \end{equation*} 
where $H_{\alpha}$ is defined in analogy to \eqref{H}, with $n$ replaced by $\alpha$; namely
    \begin{equation*}
        H_{\alpha}(t)= \l(\int_0^t \l(\frac{\tau}{A(\tau)}\r)^{\frac{1}{\alpha -1}} d\tau \r)^{\frac 1{\alpha'}} \quad \text{for $t\geq 0$,}
    \end{equation*}
with a parallel convention on the modification of $A$ near zero as in \eqref{H} in such a way that condition \eqref{convint0alpha} is fulfilled.
\\
Indeed, 
if $\Om \in \mathcal{J}_{1/{\alpha'}}$ for some $\alpha \geq n$,  then
    \begin{equation}
    \label{p1}    
   W^{1, A}(\Om) \to L^{A_\alpha}(\Om).
        \end{equation}
Furthermore, there exists a positive constant $\kappa_5=\kappa_5(\Omega)$ such that
\begin{equation}
        \label{eq:poincarealpha}
        \into A_{\alpha} \l(\frac{\kappa_5 |u- {\rm med}(u)|}{\l( \into A(|\nabla u|) \,dy\r)^{1/\alpha}} \r) dx \le \into A(|\nabla u|) \,dx 
    \end{equation}
for every $u \in V^1K^A(\Omega)$.  
In particular, 
\begin{equation}\label{embinf1}
W^{1,A}(\Om) \to L^\infty(\Om)
\end{equation}if
\begin{equation}\label{intconvalpha}
\int^\infty \l(\frac{t}{A(t)}\r)^{\frac{1}{\alpha -1}} dt < \infty.
\end{equation}
By contrast, embedding   \eqref{embinf1} is not guaranteed in the opposite regime corresponding to condition \eqref{intdivalpha}.
\\
If the latter condition  is in force, the target space in embedding \eqref{p1} admits an improvement in the 
class of  rearrangement-invariant spaces, which parallels \eqref{embol}. 
For an arbitrary domain $\Omega \in  \mathcal{J}_{1/{\alpha'}}$, the optimal target
for embeddings of $W^{1, A}(\Om)$ into rearrangement-invariant spaces is the Orlicz-Lorentz space $L(\widehat A_\alpha, \alpha)(\Omega)$, endowed with the norm defined as
\begin{equation*}
\|u\|_{L(\widehat A_\alpha, \alpha)(\Omega)} = \big\|s^{-\frac 1\alpha}
u^* (s)\big\|_{L^{\widehat A_\alpha}(0, |\Omega|)}
\end{equation*}
for $u \in \M (\Omega)$, where $\widehat A_\alpha$ is the Young function defined as in \eqref{hatAn}--\eqref{hatan}.
%
Analogously to  the embeddings presented above, 
the function $A$  is suitably modified near zero, if necessary, for condition  \eqref{convint0alpha} to be fulfilled. This ensures that the right-hand side of equation \eqref{hatan} is well defined.
\\
The   embedding in question takes the form
 \begin{equation*}
   W^{1, A}(\Om) \to L(\widehat A_\alpha, \alpha)(\Omega).
        \end{equation*}
Moreover, a positive constant $\kappa_6=\kappa_6(\Omega)$ exists such that
\begin{equation}\label{intol}
\int_0^{|\Omega|}\widehat A_\alpha \big(\kappa_6 s^{- 1/\alpha} (u-{\rm med} (u))^*(s)\big)\, ds \leq \int_\Omega  A(|\nabla u|)\, dx
\end{equation}
for every $u \in V^1K^A(\Omega)$.

A linear trace operator can be defined on any Orlicz-Sobolev space $W^{1,A}(\Omega)$ under suitable regularity assumptions on  $\Omega$. Minimal conditions on a bounded open set $\Omega$, which allow for the definition of a  trace operator,  are:
    \begin{equation*}
    \mathcal H^{n-1}(\partial \Omega)<\infty, \quad \mathcal H^{n-1}(\partial \Omega \setminus \partial ^M\Omega)=0,
        \end{equation*}
and 
 \begin{equation*} 
\min\{\hh (\partial ^M E \cap
\partial \Omega) \, , \hh ( \partial \Omega \setminus \partial ^M E)\} \leq \kappa_7 \hh
(\partial ^M E \cap \Omega)
\end{equation*}
for some positive constant $\kappa_7>0$ and every measurable set $E \subset
\Omega$. An open set fulfilling these properties will be called an \emph{admissible domain} in what follows. 
 In particular, any Lipschitz
domain is an admissible domain. In turn,  any  admissible domain belongs to the class $\mathcal J_{1/{n'}}$
\\
The assumption
that $\Omega$ be an admissible domain is necessary and sufficient
for the existence of a linear trace operator ${\rm Tr}: W^{1,1}(\Omega) \to  L^1(\partial \Omega )$   -- see  \cite[Section 5.10]{Z} or \cite[Theorem
9.5.2]{Mabook}. 
\\ The Young function $A_T$ defining the optimal Orlicz target space in Orlicz-Sobolev trace embeddings was exhibited in \cite{Cianchi-2010}, and is given by 
	\begin{equation}
	\label{AT}
A_T(t)= \int_0^t \frac{A(H_n^{-1}(\tau))}{H_n^{-1}(\tau)}\, d\tau  \quad \text{for $t \ge 0$,}
%
	\end{equation}
where  $H_n$ is as in \eqref{H} and, without loss of generality, the function $A$ is assumed to fulfill condition \eqref{intconv0}. From \cite[Theorem 3.1-(i)]{Cianchi-2010}, we have that
 \begin{equation}
        \label{trace-ineq1}
    \operatorname{Tr}: W^{1,A}(\Om) \to L^{A_T}(\rand),
    \end{equation}
the space  $L^{A_T}(\rand)$ being optimal in \eqref{trace-ineq1} among all Orlicz target spaces.
Moreover, there exists a positive constant $\kappa_8=\kappa_8(\Omega)$ such that 
\begin{equation}\label{traceint}
    \intor A_T \l(\frac{\kappa_8|\operatorname{Tr} u- {\rm med}(u)|}{\big(\into  A(|\nabla u|) \, dy \big)^{1/n}} \r) d \ch^{n-1} \le \l(\into   A(|\nabla u|) \, dx \r)^{{1}/{n'}}
   \end{equation}
for every $u \in V^1K^A(\Omega)$.
\\ In particular, under assumption \eqref{intconv},  one has that $L^{A_T}(\rand)= L^\infty(\partial \Omega)$, up to equivalent norms.
\\ Let us notice that although \cite[Theorem 3.1-(i)]{Cianchi-2010} is stated for Lipschitz domains $\Omega$, it continues to hold for any admissible domain. Indeed, its proof makes use of an interpolation argument which relies upon endpoint  boundedness properties of the operator ${\rm Tr}$, which are still enjoyed  on any admissible domain.

\section{Main results} 

Our results about Dirichlet, Neumann and Robin problems for equation \eqref{eq} are presented in Subsections \ref{secdir}, \ref{secneu} and \ref{secrob}, respectively. Each result is illustrated by an example involving equations whose principal part has a power-times-logarithmic type growth. 

\subsection{Dirichlet problems}\label{secdir}

Let  $\Om$ be an open set in $\RN$,  with $|\Omega|<\infty$, and  let $u_0 \in  V^1 K^A(\Om)$. Consider the  Dirichlet   problem
	\begin{equation}
	\label{dirichlet}
	\begin{cases}
	-\divergenz \big(\mathcal A(x, u, \nabla u)\big)= \mathcal B(x, u, \nabla u) \quad & \text{in } \Om \\
	u= u_0 \quad & \text{on } \rand,
	\end{cases}
	\end{equation}
under the growth conditions \eqref{hpA} and \eqref{hpB}.  
 A function $u \in V^1 K^A(\Om)$ is said to be a weak solution to problem \eqref{dirichlet} if $u-u_0 \in V_0^1 K^{A}(\Om)$ and 
	\begin{equation}
	\label{dir-weak-sol}
	\into \mathcal A(x, u, \nabla u) \cdot \nabla \phi \; dx= \into \mathcal B(x, u, \nabla u)\, \phi \; dx 
	\end{equation}
for every $\phi \in V_0^1 K^{A}(\Om)$. \\
Our assumptions ensuring the boundedness of any solution to problem \eqref{dirichlet} read as follows. 
The boundary datum $u_0$ is required to fulfill the condition
\begin{equation}\label{ipu0}
u_0 \in L^\infty (\Omega).
\end{equation}
As far as the functions $B, C, D, E$ appearing in equations  \eqref{hpA} and \eqref{hpB} are concerned, we assume that
\begin{equation}\label{ipyoung}
A \circ E^{-1} \quad \text{ is a Young function,}
\end{equation}
and  
there exist constants $\sigma>0$, $k>1$ and $t_0>0$  such that
\begin{equation}\label{ipB}
		B(t) \le A_n(\sigma t) \quad  \text{for   $t> t_0$,}
		\end{equation}
\begin{equation}\label{ipC}
C(t) \le \frac{A_n(\sigma t)}{t} \quad \text{for   $t> t_0$,}
\end{equation}
\begin{equation}\label{ipD}
D(t) \le \frac{1}{kt}  \l(\l(A \circ E^{-1}\r)^{\sim} \r)^{-1} \circ A_n(\sigma t) \quad \text{for   $t> t_0$,}
\end{equation}
where $A_n$ is the Sobolev conjugate of $A$ defined by \eqref{An}.
\\ Finally, we require that 
\begin{equation}\label{ipf}
f \in L^{M, \infty}(\Omega)
\end{equation}
for some Young function $M$ growing so fast   near infinity that
   \begin{align}\label{new Mn}
       \int_0A^{-1 }\Bigg(\frac{1}{s} \left(\int_0^s \widetilde A \big(\lambda r^{ 1/n} M^{-1}(1/r) \big) dr+
\int_s^{\infty} \widetilde A \big(\lambda r^{- 1/{n'}} s M^{-1}(1/s) \big) dr\right)\Bigg) \frac{ds}{s^{1/{n'}}}<\infty  
   \end{align}
for every $\lambda >0$. 

Let us emphasize that conditions \eqref{ipB}--\eqref{new Mn}  depend only on the asymptotic behaviour of the functions $A, B, C, D, E, M$ and $A_n$ near infinity. 

 In view of the Sobolev embedding \eqref{embinf}, it suffices to consider the case when the Young function $A$ fulfills condition \eqref{intdiv}, otherwise any weak solution to problem \eqref{dirichlet} is automatically bounded.

\begin{theorem}
\label{thm:dir}
Let $\Om$ be an open set in  $\RN$ with finite measure. Suppose that the Young function $A$ satisfies condition \eqref{intdiv} and that assumptions \eqref{ipu0}--\eqref{ipf} are in force. If $u$ is a weak solution to the Dirichlet problem \eqref{dirichlet}, then, $u \in L^{\infty}(\Om)$.
\end{theorem}

\begin{remark}\label{oss1}
{\rm
Under the additional assumption that $A \in \nabla_2$ near infinity, inequality \eqref{new Mn} holds for every $\lambda >0$ if and only if it just holds for some $\lambda >0$. This is a consequence of property \eqref{deltanabla} and inequality \eqref{Ak}.}
\end{remark}

\begin{remark}\label{oss2}
{\rm 
A change of variables shows that
\begin{equation}\label{alternate}
    \int_s^{\infty} \widetilde A \big(\lambda r^{-\frac 1{n'}} s M^{-1}(1/s) \big) dr= n's^{n'}M^{-1}(1/s)^{n'}\int_0^{s^{\frac 1{n}}M^{-1}(1/s)}\frac{\widetilde A (\lambda r)}{r^{1+n'}}\,dr \quad \text{for $s>0$.}
 \end{equation}
The function on the right-hand side of equation \eqref{alternate} is usually handier than the one on the left-hand side in verifying condition \eqref{new Mn}.
}    
\end{remark}

\begin{example}\label{ex:dir}
{\rm
Assume that $A$  is a Young function such that
\begin{equation}\label{exA}
A(t)  \approx t^p (\log t)^{\delta} \quad \text{near infinity,}
\end{equation}
where either $p>1$ and $\delta \in \mathbb R$, or $p=1$ and $\delta \geq 0$.
\\ Assumption \eqref{intdiv} amounts to requiring that
\begin{equation}\label{apr24}
\text{either $p< n$,    or   $p=n$ and  {$\delta \le n-1$}.}
\end{equation}
One can verify that
    $$
    A_n(t) \approx
    \begin{cases}
t^{\frac{n p}{n -p}}(\log t)^{\frac{n \delta}{n -p}}&\ \text{ if } \ 1\leq p<n \\
e^{t^{\frac{n}{n -1-\delta}}}&\ \text{ if } \ p=n\ \text{ and }\ \delta <n -1 \\
e^{e^{t^{\frac{n }{n -1}}}}&\ \text{ if } \ p=n\ \text{ and }\ \delta =n -1
\end{cases}
\qquad \text{near infinity.}
$$
Assume that the functions $B$ and $C$, appearing in conditions \eqref{hpA} and \eqref{hpB},  fulfill the inequalities
\begin{equation}\label{Bexdir}
B(\sigma t)\leq 
\begin{cases}
t^{\frac{n p}{n -p}}(\log t)^{\frac{n \delta}{n -p}}&\ \text{ if } \ 1\leq p<n \\
e^{t^{\frac{n}{n -1-\delta}}}&\ \text{ if } \ p=n\ \text{ and }\ \delta <n -1 \\
e^{e^{t^{\frac{n }{n -1}}}}&\ \text{ if } \ p=n\ \text{ and }\ \delta =n -1
\end{cases}
\qquad \text{near infinity,}\
\end{equation}
and 
\begin{equation}\label{Cexdir}
C(\sigma t)\leq 
\begin{cases}
t^{\frac{n p}{n -p}-1}(\log t)^{\frac{n \delta}{n -p}}&\ \text{ if } \ 1\leq p<n \\
e^{t^{\frac{n }{n -1-\delta}}}&\ \text{ if } \ p=n\ \text{ and }\ \delta <n -1 \\
e^{e^{t^{\frac{n }{n -1}}}}&\ \text{ if } \ p=n\ \text{ and }\ \delta =n -1
\end{cases}
\qquad \text{near infinity,}
\end{equation}
for some constant $\sigma >0$. Also, suppose, for simplicity, that $D=E=0$.
\\ Let $f\in L^{M, \infty}(\Om)$, where $M$ is a Young function such that
\begin{equation}\label{exPhi}
M(t) \approx t^q (\log t)^{\beta} \quad \text{near infinity,}
\end{equation}
and  either $q>1$ and $\beta \in \mathbb R$, or $q=1$ and $\beta> 0$.
\\
By Theorem \ref{thm:dir},  
any weak solution  $u$ to  the Dirichlet problem \eqref{dirichlet} is bounded, provided that, in addition, condition \eqref{new Mn} is satisfied.
Since
	\begin{equation*}	    
	\widetilde A(t) \approx t^{p'} (\log t)^{-\frac{\delta}{p-1}}, \qquad  A^{-1}(t) \approx t^{1/p} (\log t)^{-\delta/p}, \qquad M^{-1}(t) \approx t^{1/q} (\log t)^{-\beta/q}\quad \text{near infinity,}
	\end{equation*}
this condition corresponds to demanding that
$$
\begin{cases}
	 \displaystyle \text{either } q> \frac{n}{p}, \\
	\text{ or } \,\,\displaystyle q= \frac{n}{p}>1 \,\,\text{ and }\,\, \beta p+ n \delta> n(p-1),
	  \\
	\text{ or }\,\,
	\displaystyle q= \frac{n}{p}=1, \, \delta \le n-1 \,\, \text{and} \,\,  \beta+\delta>n-\frac{1}{n}. \\
    \end{cases}
	$$
}
\end{example}

\medskip

\subsection{Neumann problems}\label{secneu}

Given  a bounded open set  $\Omega$ in $\RN$,   consider the  Neumann problem
	\begin{equation}
	\label{neumann}
	\begin{cases}
	-\divergenz \big(\mathcal A(x, u, \nabla u)\big)= \mathcal B(x, u, \nabla u) \quad & \text{in } \Om \\
	\mathcal A(x, u, \nabla u) \cdot \nu = 0 \quad &\text{on $\partial \Omega$,}
	\end{cases}
	\end{equation}
under the growth conditions \eqref{hpA} and \eqref{hpB}.  
 A function $u \in V^1 K^A(\Om)$ is said to be a weak solution to problem \eqref{neumann} if 
	\begin{equation}
	\label{neu-weak-sol}
	\into \mathcal A(x, u, \nabla u) \cdot \nabla \phi \; dx= \into \mathcal B(x, u, \nabla u)\, \phi \; dx 
	\end{equation}
for every $\phi \in V^1 K^{A}(\Om)$.

The conditions to be imposed on the functions appearing on the right-hand sides of equations \eqref{hpA} and \eqref{hpB}  depend on the degree of regularity of the domain $\Omega$. The latter is assumed to satisfy
\begin{equation*} 
\Omega \in \mathcal J_{1/\alpha'}
\end{equation*}
for some  exponent $\alpha \geq n$. Here,  $\mathcal J_{1/\alpha'}$ denotes the 
 Maz'ya class defined according to inequality \eqref{isop}.
\\ 
Besides hypothesis \eqref{ipyoung}, the functions $B, C, D, E$ are requested to fulfill bounds parallel to \eqref{ipB}--\eqref{ipD}, with the function $A_n$ replaced by the Sobolev conjugate $A_\alpha$ entering embedding \eqref{p1}. Namely:
\begin{equation}\label{ipBn}
		B(t) \le A_\alpha(\sigma t) \quad  \text{for   $t> t_0$,}
		\end{equation}
\begin{equation}\label{ipCn}
C(t) \le \frac{A_\alpha(\sigma t)}{t} \quad \text{for   $t> t_0$,}
\end{equation}
\begin{equation}\label{ipDn}
D(t) \le \frac{1}{kt}  \l(\l(A \circ E^{-1}\r)^{\sim} \r)^{-1} \circ A_\alpha(\sigma t) \quad \text{for   $t> t_0$,}
\end{equation}
for some constants $\sigma>0$, $k>1$ and $t_0\geq 0$.
\\
 Similarly,  condition \eqref{new Mn} on the Young function $M$ defining an admissible ambient weak Orlicz space $L^{M, \infty}(\Omega)$ for $f$ turns into 
   \begin{align}\label{new M}
       \int_0A^{-1 }\Bigg(\frac{1}{s} \left(\int_0^s \widetilde A \big(\lambda r^{ 1/\alpha} M^{-1}(1/r) \big)dr +
\int_s^{\infty} \widetilde A \big(\lambda r^{- 1/{\alpha'}} s M^{-1}(1/s) \big) dr\right)
\Bigg) \frac{ds}{s^{1/\alpha'}}<\infty  
   \end{align}
for every $\lambda >0$.
\\ Owing to embedding \eqref{embinf1}, which holds  under assumption \eqref{intconvalpha},  in our result on Neumann problems it suffices to focus on the case when the function $A$ fulfills the complementary condition \eqref{intdivalpha}.

\begin{theorem}
\label{thm:neu}
Let $\Om$ be an open bounded set in  $\RN$ from the class $\mathcal J_{1/\alpha'}$, for some $\alpha \geq n$. Let $A$ be a Young function satisfying condition \eqref{intdivalpha}. Suppose that assumptions \eqref{ipyoung} and \eqref{ipBn}--\eqref{ipDn}  are in force, and that condition \eqref{ipf} holds for some Young function $M$ satisfying \eqref{new M}. If $u$ is a weak solution to the Neumann problem \eqref{neumann}, then, $u \in L^{\infty}(\Om)$.
\end{theorem}

\begin{remark} 
{\rm
Observations parallel to Remarks \ref{oss1} and \ref{oss2} apply with regard to 
Theorem \ref{thm:neu}.  
In particular, the identity
    \begin{equation}\label{june1}
    \int_s^{\infty} \widetilde A \big(\lambda r^{- 1/{\alpha'}} s M^{-1}(1/s) \big) dr= \alpha ' s^{\alpha'}M^{-1}(1/s)^{\alpha'}\int_0^{s^{ 1/{\alpha}}M^{-1}(1/s)}\frac{\widetilde A (\lambda r)}{r^{1+\alpha'}}\,dr \quad \text{for $s >0$}
   \end{equation}
can be of use in connection with condition \eqref{new M}.}
\end{remark}

\begin{example}
{\rm Assume that $A$ is a Young function as in  \eqref{exA}. In view of condition \eqref{intconvalpha}, assumptions \eqref{apr24} have to be replaced by
%
\begin{equation*}
\text{either $p=1$ and $\delta \geq 0$, or $1<p< \alpha$, or $p=\alpha$ and  {$\delta \le \alpha -1$}.}
\end{equation*}
Moreover,  
  $$
    A_\alpha(t) \approx
    \begin{cases}
t^{\frac{\alpha p}{\alpha -p}}(\log t)^{\frac{\alpha \delta}{\alpha -p}}&\ \text{ if } \ 1\leq p<\alpha\\
e^{t^{\frac{\alpha }{\alpha -1-\delta}}}&\ \text{ if } \ p=\alpha\ \text{ and }\ \delta <\alpha -1\\
e^{e^{t^{\frac{\alpha }{\alpha -1}}}}&\ \text{ if } \ p=\alpha\ \text{ and }\ \delta =\alpha -1
\end{cases}
\qquad \text{near infinity.}
  $$
Assume that the functions $B$ and $C$ in conditions \eqref{hpA} and \eqref{hpB} satisfy the inequalities
    $$
    B(\sigma t)\leq 
    \begin{cases}
t^{\frac{\alpha p}{\alpha -p}}(\log t)^{\frac{\alpha \delta}{\alpha -p}}&\ \text{ if } \ 1\leq p<\alpha \\
e^{t^{\frac{\alpha }{\alpha -1-\delta}}}&\ \text{ if } \ p=\alpha\ \text{ and }\ \delta <\alpha -1 \\
e^{e^{t^{\frac{\alpha }{\alpha -1}}}}&\ \text{ if } \ p=\alpha\ \text{ and }\ \delta =\alpha -1
    \end{cases} \quad \text{near infinity,}
$$
and
$$
    C(\sigma t)\leq 
    \begin{cases}
t^{\frac{\alpha p}{\alpha -p}-1}(\log t)^{\frac{\alpha \delta}{\alpha -p}}&\ \text{ if } \ 1\leq p<\alpha\\
e^{t^{\frac{\alpha }{\alpha -1-\delta}}}&\ \text{ if } \ p=\alpha\ \text{ and }\ \delta <\alpha -1\\
e^{e^{t^{\frac{\alpha }{\alpha -1}}}}&\ \text{ if } \ p=\alpha\ \text{ and }\ \delta =\alpha -1
\end{cases}
\quad \text{near infinity,}
    $$
for some constant $\sigma >0$. Moreover, let $D=E=0$.
\\ Let $f\in L^{M, \infty}(\Om)$, where $M$ is a function obeying \eqref{exPhi}. 
\\ By Theorem \ref{thm:neu},  in view of condition \eqref{new M} any weak solution  $u$ to  problem \eqref{neumann} is bounded provided that 
 $$
	\begin{cases}
    \displaystyle \text{either } q> \frac{\alpha}{p}, \\
	\text{ or } \displaystyle q= \frac{\alpha}{p}>1 \text{ and } \beta p+ \alpha \delta> \alpha(p-1),\\
\text{ or }
	\displaystyle q= \frac{\alpha}{p}=1, \delta \le \alpha-1\ \text{and}\   \beta+\delta>\alpha-\frac{1}{\alpha}. \\
    \end{cases}
	$$
}

\end{example}

\medskip

\subsection{Robin problems}\label{secrob}

Assume now that  $\Omega$ is an admissible  domain in $\RN$, as defined in Section \ref{spaces}.
Here we focus on the Robin problem
	\begin{equation}
	\label{robin}
	\begin{cases}
	-\divergenz \big(\mathcal A(x, u, \nabla u)\big)= \mathcal B(x, u, \nabla u) \quad & \text{in } \Om \\
	\mathcal A(x, u, \nabla u) \cdot \nu =  \mathcal C(x, u) \quad & \text{on $\partial \Omega$,}
	\end{cases}
	\end{equation}
under the growth conditions \eqref{hpA}, \eqref{hpB} and \eqref{hpC}. 
 A function $u \in V^1 K^A(\Om)$ is said to be a weak solution to problem \eqref{robin} if 
	\begin{equation}
	\label{weak-sol}
	\into \mathcal A(x, u, \nabla u) \cdot \nabla \phi \; dx= \into \mathcal B(x, u, \nabla u)\, \phi \; dx + \intor \mathcal C(x, u)\, \phi \; d \ch^{n-1} 
	\end{equation}
for every $\phi \in V^1 K^{A}(\Om)$.   For simplicity of notation, here, and in what follows, we just write $u$ instead of ${\rm Tr} \,u$ in integrals over $\partial \Omega$.
\\ Our boundedness criterion for solutions to problem \eqref{robin} requires the same assumptions on the functions $B, C, D, E, M$ as those for the Dirichlet problem \eqref{dirichlet} and, of course,  also entails bounds in a similar vein for the functions $F$ and $g$ involved in condition \eqref{hpC}. The maximal admissible growth of the function $F$ depends on the Sobolev trace  conjugate $A_T$ of $A$, defined by \eqref{AT}, and reads
\begin{equation} \label{dis-AT}
			F(t) \le \frac{A_T(\sigma t)}{t}  \quad \text{for $t>t_0$,}
			\end{equation}
for some constants $\sigma>0$ and  $t_0\geq  0$. 
 The integrability condition on the function $g$ takes the form
\begin{equation}\label{ipg}
g \in L^{N, \infty}(\partial \Omega),
\end{equation}
for some Young function $N$ such that
	\begin{equation}
	\label{int G} \int_0 A^{-1}\bigg(\frac {1}{s^{n'}} \int_0^s \widetilde A \bigg(\frac 1r \int_0^r
 N^{-1}(1/\rho) \,d\rho\bigg) r^{\frac 1{n-1}}\, dr\bigg) \frac{ds}{s^{\frac {n-2}{n-1}}}<\infty.
%
%
	\end{equation}
 Finally, unlike the results for Dirichlet and Neumann problems presented above, the function $A$ is required to additionally satisfy the $\nabla_2$-condition near infinity, which, loosely speaking, prevents $A$ from having  an almost linear growth. 
 This is due to the fact that the   trace  conjugate $A_T$ of $A$ boils down to $A$ itself when the latter grows linearly. Indeed, classically $A_T=A$ if $A(t)=t$ for $t \geq 0$.

\begin{theorem}
\label{thm:rob}
Let $\Om$ be an admissible domain in $\RN$. Assume that $A\in \nabla_2$ near infinity and satisfies condition \eqref{intdiv}. Suppose
 that assumptions \eqref{ipu0}--\eqref{ipf} and \eqref{dis-AT}--\eqref{ipg} 
 are in force.  If $u$ is a weak solution to the Robin problem \eqref{robin}, then, $u \in L^{\infty}(\Om)$.
\end{theorem}

\begin{example}
{\rm Let $\Omega$ be an admissible domain.  Hence, $\Om\in \mathcal J_{1/n'}$. Assume that $A$  is as in Example \ref{ex:dir}, save that,  in view of the requirement that $A\in \nabla_2$ near infinity, now only the alternative $p>1$ and $\delta \in \mathbb R$ is admissible.  By embedding \eqref{embinf1}, with $\alpha =n$, we may limit ourselves to considering exponents $p$ and $\delta$ fulfilling condition \eqref{apr24}.
%
%
\\ One can show that
    $$
    A_T(t) \approx
    \begin{cases}
t^{\frac{p( n-1)}{n -p}}(\log t)^{\frac{ \delta(n-1)}{n -p}}&\ \text{ if } \ 1< p<n \\
e^{t^{\frac{n}{n -1-\delta}}}&\ \text{ if } \ p=n\ \text{ and }\ \delta <n -1\\
e^{e^{t^{\frac{n }{n -1}}}}&\ \text{ if } \ p=n\ \text{ and }\ \delta =n -1
\end{cases}
\qquad \text{near infinity.}
  $$
Therefore, assumption 
 \eqref{dis-AT} reads
\begin{equation}
 \label{hptrace}
    F(\sigma t)\leq 
    \begin{cases}
t^{\frac{p( n-1)}{n -p}-1}(\log t)^{\frac{ \delta(n-1)}{n -p}}&\ \text{ if } \ 1< p<n \\
e^{t^{\frac{n}{n -1-\delta}}}&\ \text{ if } \ p=n\ \text{ and }\ \delta <n -1 \\
e^{e^{t^{\frac{n }{n -1}}}}&\ \text{ if } \ p=n\ \text{ and }\ \delta =n -1
\end{cases}
\qquad \text{near infinity,}
 \end{equation}
for  some constant $\sigma >0$.
\\ Suppose that $f \in L^{M,\infty}(\Om)$, where $M$ is as in Example \ref{ex:dir}. Also, assume that 
$g \in L^{N,\infty}(\partial \Om)$, where $N$ is a Young function such that
\begin{equation*}
N(t)  \approx t^r (\log t)^{\gamma} \quad \text{near infinity,}
\end{equation*}
 where either  $r> 1$  and $\gamma \in \mathbb R$, or  $r=1$ and $\gamma \ge 0$. Condition \eqref{int G} reduces to requiring that
\begin{equation*} 
	\begin{cases}
	\text{either } \displaystyle  r> \frac{n-1}{p-1},  \\
	 \text{or }  \displaystyle  r= \frac{n-1}{p-1}> 1, \ \text{and} \ \gamma (p-1)+ \delta(n-1)> (n-1)(p-1), \\
	 \text{or } \displaystyle  r= 1,\ p=n,\ \delta\leq n-1,\ \text{and}\ \gamma + \delta> n+1-\frac1n\,.
	\end{cases}
	\end{equation*}
Under these assumptions on $f$ and $g$, if $D=E=0$ and inequalities \eqref{Bexdir}, \eqref{Cexdir}, \eqref{hptrace} are satisfied, then Theorem \ref{thm:rob} ensures that any weak solution $u$ to the Robin problem \eqref{robin} is bounded.
}

\end{example}

\section{Proofs of Theorems \ref{thm:dir} and \ref{thm:neu}} 

A detailed proof will be provided for Theorem \ref{thm:neu}, dealing with Neumann problems. The argument supporting Theorem \ref{thm:dir}, which concerns Dirichlet problems, is analogous, and even simpler. We thus limit ourselves to sketching it, and to highlighting the necessary modifications.

We begin with a technical lemma, dealing with the integrand appearing in condition \eqref{new M}. In preparation for its statement and proof, a couple of auxiliary functions are introduced.

Let $A$ and $M$ be Young functions fulfilling condition \eqref{new M}. Define, for $\lambda >0$, the function $\Psi_\lambda : (0, \infty) \to [0, \infty]$ as
 \begin{equation}
	\label{N}
	\Psi_\lambda(s)= \frac{1}{s} \bigg(\int_0^s \widetilde A \l(\lambda r^{ 1/\alpha}  M^{-1}(1/r) \r) dr
	+\int_s^{\infty} \widetilde A \l(\lambda r^{- 1/{\alpha'}} s M^{-1}(1/s) \r) dr \bigg) \quad \text{for $s>0$.}
	\end{equation}
Hence, condition \eqref{new M} takes the form
\begin{align}\label{new N}
       \int_0s^{- 1/{\alpha'}}A^{-1 }\l(\Psi_\lambda(s)\r)ds<\infty
   \end{align}	
for every $\lambda >0$.
Next, let $\Phi_\lambda\colon [0, \infty) \to [0, \infty]$ be the function defined by
     \begin{equation}
	\label{def-M}
	\Phi_\lambda (s)= s A^{-1} \l(\Psi_\lambda(s)\r) \quad \text{for $s>0$.}
	\end{equation}



\begin{lemma}
\label{prop:M}
Let $A$ and $M$ be Young functions fulfilling condition \eqref{new M} for some $\alpha \geq n$.    Assume that $M$ is differentiable and
\begin{equation}\label{june5}
 \Big(\frac {M(t)}t \Big)'>0 \qquad \text{for $t >0$.}
\end{equation}
Let $\lambda >0$ and let $\Phi_\lambda$ be the function given by \eqref{def-M}. Then,
    \begin{itemize}
   \item[(i)] $\lim_{s\to 0^+} \Phi_\lambda(s)=0$;
        \item[(ii)] Either there exists $\overline s>0$ such that $\Phi_\lambda $ vanishes in $s\in(0, \overline s)$, or there   exists $\overline s >0$ such that
 $\Phi_\lambda$ is strictly increasing in $(0, \overline s)$;
%
        \item[(iii)] Condition \eqref{new N} is equivalent to 
\begin{equation*}
\int_0 \frac{1}{\Phi_\lambda^{-1}(\tau)^{\frac 1{\alpha'}}} \; d\tau< \infty.
\end{equation*}
    \end{itemize}
     
\end{lemma}

\begin{proof} {Property (i) is a consequence of (ii) and of condition \eqref{new N}.} 
As for property (ii),
let us define the function $  T_\lambda :(0, \infty) \to [0, \infty]$ as
	\begin{align}\label{Tlambda}
    T_\lambda(s)&= \int_0^s \widetilde A \l(\lambda{r^{{1}/{\alpha}}}  M^{-1}(1/r) \r) dr+ 
  \alpha ' s^{\alpha'}M^{-1}(1/s)^{\alpha'}\int_0^{s^{ 1/{\alpha}}M^{-1}(1/s)}\frac{\widetilde A (\lambda\varrho)}{\varrho^{1+\alpha'}}\,d\varrho \quad \text{for $s>0$.}
    \end{align}
By equations \eqref{N} and \eqref{june1},
    \begin{align}
     \Psi_\lambda(s)= \frac{1}{s} T_\lambda(s) \quad \text{for $s>0$.} \label{def-N1}
	\end{align}
Thanks to assumption \eqref{new N}, there exists $s_0>0$ such that $T_\lambda (s)<\infty$ for $s\in (0, s_0)$.
\\ We claim that either  $T_\lambda (s)>0$, or  $T_\lambda (s)=0$ for  every
$s\in (0, s_1)$, for some $s_1 \in (0, s_0)$. Indeed, the first alternative trivially holds if $\widetilde A(t)>0$ for every $t>0$. Assume now that there exists $t_0>0$ such that
$\widetilde A (t)=0$ for $t \in [0, t_0]$ and $\widetilde A (t)>0$ for $t >t_0$.  We may clearly assume that $s_0>t_0$.  If   there exists $s_1 \in (0, t_0)$ such that 
$$ \lambda{r^{{1}/{\alpha}}}  M^{-1}(1/r)\leq t_0  \quad \text{for $r\in (0,s_1)$,}
	    $$
then, since the function  $s M^{-1}(1/s)$ is non-decreasing, we also have that
$$ \lambda{r^{-{1}/{\alpha'}}}  s M^{-1}(1/s)\leq t_0  \quad \text{for $s\in (0,s_1)$ and $ r \in  (s_1, \infty)$.}
	    $$
Hence, owing to equation \eqref{june1},
$T_\lambda (s)=0$ for $s \in (0, s_1)$. On the other hand, if for every $s \in (0, t_0)$ there exists  $r\in (0,s)$ such that
$$ \lambda{r^{{1}/{\alpha}}}  M^{-1}(1/r)>t_0,
	    $$
then, by the continuity of the function $ \lambda{r^{{1}/{\alpha}}}  M^{-1}(1/r)$, the first integral on the right-hand side of equation \eqref{Tlambda} is positive, whence $T_\lambda (s) >0$ for $s \in (0, t_0)$.
Inasmuch as $\Phi_\lambda$ vanishes if and only if $T_\lambda$ vanishes, and $\Phi_\lambda =\infty$  if and only if $T_\lambda =\infty$, we have shown that  either there exists $\overline s >0$ such that $\Phi_\lambda(s) = 0 $ for $s\in(0, \overline s)$, or there   exists $\overline s >0$ such that $0<\Phi_\lambda (s)<\infty $ for $s\in(0, \overline s)$.
\\ In order to show that, in the latter case, $\Phi_\lambda$ is strictly increasing, it suffices to show that
\begin{equation}\label{M'}
\Phi'_\lambda(s) >0 \quad \text{for $s\in (0, \overline s)$.}
\end{equation}
To this purpose, observe that
	  \begin{align}\label{T'bis}
	&T_\lambda'(s) \\
	& =  \widetilde A \l(\lambda{s^{{1}/{\alpha}}}  M^{-1}(1/s) \r) + \alpha ' \frac{s^{{1}/{\alpha'}}}{ M^{-1}(1/s)} \widetilde A \l(\lambda{s^{{1}/{\alpha}}}  M^{-1}(1/s) \r)\Big(\frac 1\alpha s^{-{1}/{\alpha'}}M^{-1}(1/s)  + s^{1/\alpha}\big(M^{-1}(1/s)\big)'\Big) \nonumber
\\ \nonumber & \quad + \alpha ' \Big(s^{\alpha'}M^{-1}(1/s)^{\alpha'}\big)' \int_0^{s^{ 1/{\alpha}}M^{-1}(1/s)}\frac{\widetilde A (\lambda\varrho)}{\varrho^{1+\alpha'}}\,d\varrho
\\ \nonumber & = \alpha' \widetilde A \l(\lambda{s^{{1}/{\alpha}}}  M^{-1}(1/s) \r) + \alpha '\frac{s \big(M^{-1}(1/s)\big)'}{M^{-1}(1/s)}\widetilde A \l(\lambda{s^{{1}/{\alpha}}}  M^{-1}(1/s) \r)
\\ \nonumber & \quad +
 \alpha ' \Big(s^{\alpha'}M^{-1}(1/s)^{\alpha'}\Big)' \int_0^{s^{ 1/{\alpha}}M^{-1}(1/s)}\frac{\widetilde A (\lambda\varrho)}{\varrho^{1+\alpha'}}\,d\varrho
\quad \quad \quad \quad \quad \quad \quad \quad  \text{for $s\in (0, \overline s)$.} 
	\end{align}
On the other hand,
\begin{align*} 
\frac{s \big(M^{-1}(1/s)\big)'}{M^{-1}(1/s)}= \frac{-1/s \big(M^{-1}\big)'(1/s)}{M^{-1}(1/s)}\geq -1 \quad \text{for $s>0$,}
\end{align*}
where the inequality holds since $M^{-1} : [0, \infty) \to [0, \infty)$ is a concave function. Moreover, thanks to assumption \eqref{june5},
\begin{equation}\label{june4}
\Big(s^{\alpha'}M^{-1}(1/s)^{\alpha'}\Big)'>0 \quad \text{for $s\in (0, \overline s)$.}
\end{equation}
Combining inequalities  \eqref{T'bis}--\eqref{june4} tells us that
\begin{equation}\label{T'}
T_\lambda'(s) >0 \qquad \text{for $s\in (0, \overline s)$.} 
\end{equation}
Next, set
	\[
	\mathcal E = \{s\in (0, \overline s) : \, \Psi_\lambda'(s)= 0\}.
	\]
The set $\mathcal E$ is closed, inasmuch as the function $\Psi_\lambda$ is continuously differentiable. Hence, the set $ (0, \overline s) \setminus \mathcal E$ is open and therefore there exists a set $K \subset \N $ such that
	\[
	 (0, \overline s) \setminus \mathcal E= \bigcup_{i \in K} (a_i, b_i),
	\]
and $K= I \cup J$, where the sets  $I $ and $J$ are such that   $\Psi_\lambda'>0$   in $(a_i, b_i)$ for $i \in I$, and $\Psi_\lambda'<0$  in $(a_i, b_i)$ for  $i \in J$. 
\\ If  $s \in (a_i, b_i)$ with $i \in I$,  then inequality  \eqref{M'} is a straightforward consequence of definition \eqref{def-M}. 
\\ If  $ s \in (a_i, b_i)$ with $i \in J$, then coupling \eqref{def-M} with \eqref{def-N1} tells us that
$$ 
\Phi_\lambda(s)= \frac{A^{-1}(\Psi_\lambda(s))}{\Psi_\lambda(s)} T_\lambda(s),$$ whence inequality \eqref{M'} follows, thanks to equation \eqref{T'} and property \eqref{mono}.
\\ Suppose now that $s \in \mathcal E$. Differentiating the right-hand side of equation \eqref{def-M}  yields
    \[
    \Phi_\lambda'(s)= A^{-1}(\Psi_\lambda(s)),
    \]
thus establishing inequality \eqref{M'} also in this case.
\\ Finally,   assertion (iii) is a consequence of the following chain, which relies upon  Fubini's theorem and a change of variables:
    \begin{align}  \label{apr20}
      \int_0 \frac{d\tau}{(\Phi_\lambda^{-1}(\tau))^{1/\alpha'}}& = \int_0 \l(\alpha '\int_{\Phi_\lambda^{-1}(\tau)}^\infty s^{-\frac 1{\alpha'}- 1} \;ds \r) d\tau= \int_0 s^{-\frac 1{\alpha'}- 1} \l(\int_0^{\Phi_\lambda(s)} d\tau \r) ds  \\ \nonumber &=  \int_0 s^{-\frac 1{\alpha'}- 1}  \Phi_\lambda(s)\, ds= \int_0 s^{-\frac 1{\alpha'}} A^{-1} (\Psi_\lambda(s)) \; ds.
      \end{align} 
\end{proof}

We are now in a position to prove Theorem \ref{thm:neu}.
\begin{proof}[Proof of Theorem \ref{thm:neu}]  Let us preliminarily observe that, since assumptions \eqref{ipBn}--\eqref{ipDn} are only required for large values of $t$, and assumption \eqref{new M} also only depends on   the functions $A$ and $M$ for large values of their argument, we may assume, without loss of generality, that condition \eqref{convint0alpha} is in force. Also, on replacing $M$, if necessary, by an equivalent Young function, we may assume that the function $M$ fulfills the hypotheses of Lemma \ref{prop:M}.

Let $u$ be a weak solution to problem \eqref{neumann} and assume,  by contradiction, that $\displaystyle \esssup |u|= \infty$. 
Let $t_0$ be the constant appearing in conditions \eqref{ipBn}--\eqref{ipDn}.
Fix any $t >\max\{t_0, {\rm med} (|u|)\}$  and consider  the function 
\begin{equation}\label{phi}
\phi= \sgn(u) (|u|- t)_+ ,
\end{equation}
where the subscript $\lq\lq + "$ stands for  positive part. Standard results on truncations of weakly differentiable functions and our assumption that $u\in V^1K^A(\Omega)$ ensure that  $\phi \in V^1 K^A(\Om)$ as well. Moreover, 
$\nabla \phi = \chi_{\{|u|>t\}}\nabla u$  a.e. in $\Omega$.
The use of this test function \eqref{phi}  in the weak formulation \eqref{neu-weak-sol}  of the Neumann problem  \eqref{neumann} yields
	\begin{equation}
	\label{dir-1}
	\int_{\Om_t} \mathcal A(x, u, \nabla u) \cdot \nabla u \; dx= \int_{\Om_t} \mathcal B(x, u, \nabla u) \sgn(u) (|u|-t) \;dx,
	\end{equation}
where we have set
\begin{equation}\label{omegat}
\Om_t = \{x\in \Omega: |u(x)|>t\}  \quad \text{for $t>0$.}
\end{equation}
Let us also define  the distribution function $\mu : [0, \infty) \to [0, \infty)$ of $u$ as 
\begin{equation}\label{mu}
\mu (t) = |\Om_t| \quad \text{for $t \geq 0$.}
\end{equation}
From inequalities \eqref{hpA} and \eqref{ipBn} one can deduce that
	\begin{equation}
	\label{dir-1.1}
	\intot \mathcal A(x, u, \nabla u) \cdot \nabla u \, dx  \ge \intot A(|\nabla u|)- B(|u|) \, dx   \ge \intot A(|\nabla u|) \,dx- \intot A_\alpha (\sigma|u|) \, dx.
	\end{equation}
On the other hand, owing to inequality \eqref{hpB} we have that
	\begin{align}
	\label{dir-2}
	\int_{\Om_t} \mathcal B(x, u, \nabla u) \sgn(u) (|u|- t) \; dx& \le \int_{\Om_t} f(x) (|u|- t) \; dx+ \int_{\Om_t} C(|u|) (|u|- t) \; dx \\
	& \qquad+ \int_{\Om_t} D(|u|) E(|\nabla u|) (|u|- t) \; dx. \nonumber
	\end{align}
We now   estimate the integrals on the 
right-hand side of \eqref{dir-2}.  \\ Let $\eps \in (0,1)$ to be chosen later, and let $\kappa_6$ be the constant appearing in inequality \eqref{intol}. An application of the 
 Hardy-Littlewood inequality \eqref{HL} and of inequality \eqref{youngcianchi}  yields
	\begin{align}
	& \int_{\Om_t} f(x)(|u|-t) \; dx \label{dir-2.1}  \le \int_0^{\mu(t)} f^*(r) (|u|- t)_+^*(r) \; dr \\ \nonumber 
	& \le \kappa_1\eps \int_0^{ \mu(t)} \widehat A_\alpha\l(\kappa_6 r^{-\frac 1\alpha}  (|u|- t)_+^*(r) \r) dr+ \kappa_1\eps \int_0^{\infty} \widetilde A\l(\frac{1}{\kappa_6\eps} r^{-\frac1{\alpha'}} \int_0^r f^*(\rho) \chi_{[0, \mu(t)]} (\rho) \;d\rho \r) dr, \nonumber
	\end{align}
where $\kappa_1$ denotes the constant from inequality \eqref{youngcianchi}.
Owing to inequality \eqref{intol} applied to the function $(|u|- t)_+$ we have that
	\begin{align}
	\label{dir-2.1.1}
	\int_0^{ \mu(t)} \widehat A_\alpha\l(\kappa_6 r^{- 1/\alpha}  (|u|- t)_+^*(r)\r) dr \le   \intot A(|\nabla u|) \,dx.
	\end{align}
Notice that here we have made use of the fact that ${\rm med}((|u|-t)_+)=0$, since we are assuming that $t> {\rm med}(|u|)$.
On the other hand, 
	\begin{align}
	\int_0^{\infty}& \widetilde A\l(\frac{1}{\kappa_6\eps} r^{- 1/{\alpha'}} \int_0^r f^*(\rho) \chi_{[0, \mu(t)]} (\rho) \,d\rho \r) dr 	\label{dir-2.2}
\\
	& = \int_0^{\mu(t)} \widetilde A\l(\frac{1}{\kappa_6\eps}r^{- 1/{\alpha'}} \int_0^r f^*(\rho)\, d\rho\r) dr+ \int_{\mu(t)}^{\infty}\widetilde A \l(\frac{1}{\kappa_6\eps}r^{- 1/{\alpha'}}\int_0^{\mu(t)} f^*(\rho) \,d\rho \r) dr. \nonumber
	\end{align}
By the definition of the weak Orlicz norm,
	\begin{align}
	\int_0^{\mu(t)} \widetilde A \l(\frac{1}{\kappa_6\eps}r^{-1/{\alpha'}} \int_0^r f^*(\rho) \,d\rho \r) dr&
	=\int_0^{\mu(t)} \widetilde A \l(\frac{1}{\kappa_6\eps}r^{ 1/\alpha} f^{**}(r) \r) dr 	\label{dir-2.2.1}\\ 
	& \le \int_0^{\mu(t)} \widetilde A \l(\frac{\|f\|_{L^{M,\infty}(\Om)}}{\kappa_6\eps}   r^{ 1/\alpha} M^{-1} \l({1}/{r}\r) \r) dr, \nonumber
	\end{align}
and
	\begin{align}
	\int_{\mu(t)}^{\infty} \widetilde A \l(\frac{1}{\kappa_6\eps}r^{-1/{\alpha'}} \int_0^{\mu(t)} f^*(\rho) \,d\rho \r) dr &
	= \int_{\mu(t)}^{\infty} \widetilde A \l(\frac{1}{\kappa_6\eps}r^{-1/{\alpha'}} \mu(t) f^{**}(\mu(t))  \r) dr  	\label{dir-2.2.2} \\
	& \le \int_{\mu(t)}^{\infty} \widetilde A \l(\frac{\|f\|_{L^{M,\infty}(\Om)}}{\kappa_6\eps} r^{-1/{\alpha'}} \mu(t)  M^{-1} \l({1}/{\mu(t)}\r) \r) dr. \nonumber
	\end{align}
Combining inequalities \eqref{dir-2.1}--\eqref{dir-2.2.2} yields
	\begin{align}
	\label{dir-2.4}
	 \int_{\Om_t} f(x)(|u|-t) \; dx
	& \le \eps \kappa_1 \int_{\Om_t} A(|\nabla u|) \;dx+ \kappa_1 \int_0^{\mu(t)} \widetilde A \l(\frac{\|f\|_{L^{M,\infty}(\Om)}}{\kappa_6\eps} r^{ 1/\alpha}   M^{-1} \l({1}/{r} \r) \r) dr \\
	& \quad + \kappa_1  \int_{\mu(t)}^{\infty} \widetilde A \l(\frac{\|f\|_{L^{M,\infty}(\Om)}}{\kappa_6\eps} r^{-1/{\alpha'}} \mu(t)  M^{-1} \l({1}/{\mu(t)}\r) \r) dr. \nonumber
	\end{align}
Next, assumption \eqref{ipCn} implies that
	\begin{equation}
	\label{2.5}
	\begin{split}
	& \int_{\Om_t} C(|u|)(|u|- t) \; dx 
	 \le \intot \frac{A_\alpha(\sigma|u|)}{|u|}(|u|- t) \; dx\leq  \intot A_\alpha(\sigma|u|) \; dx.
	\end{split}
	\end{equation}
The last integral on the right-hand side of \eqref{dir-2} can be estimated as follows: 
	\begin{align}
	\label{dir-2.6}
	& \int_{\Om_t} D(|u|)E(|\nabla u|) (|u|- t) \; dx \\
	& \le \int_{\Om_t} E(|\nabla u|) \frac{1}{k|u|} \l(\l(A \circ E^{-1} \r)^{\sim}\r)^{-1} \circ A_\alpha(\sigma |u|)  (|u|- t) \; dx \nonumber \\
	& \le \int_{\Om_t} \frac{E(|\nabla u|)}{k}   \l(\l(A \circ E^{-1} \r)^{\sim}\r)^{-1}(A_\alpha(\sigma |u|))\;dx \nonumber \\
	& \le \int_{\Om_t} (A \circ E^{-1}) \l(\frac{E(|\nabla u|)}{k} \r)+ (A \circ E^{-1})^{\sim} \l((A \circ E^{-1})^{\sim} \r)^{-1} \circ A_\alpha(\sigma |u|) \; dx \nonumber \\
	&\le \int_{\Om_t} \frac{1}{k} A\l(|\nabla u| \r)+  A_\alpha(\sigma|u|) \; dx, \nonumber 
	\end{align}
where the first inequality holds by assumption \eqref{ipDn}, the third one by Young's inequality, and the last one by inequality \eqref{Ak} applied with $A$ replaced by the Young function $A\circ E^{-1}$.
\\ Equations \eqref{dir-2} and \eqref{dir-2.4}--\eqref{dir-2.6} imply that
\begin{align}
	\label{dir-2.7new}
    & \int_{\Om_t} \mathcal B(x, u, \nabla u) \sgn(u) (|u|- t) \; dx  \\
    & \le \Big(\eps \kappa_1 + \frac 1k\Big) \int_{\Om_t} A(|\nabla u|) \;dx+ \kappa_1 \int_0^{\mu(t)} \widetilde A \l(\frac{\|f\|_{L^{M,\infty}(\Om)}}{\kappa_6\eps} r^{ 1/\alpha}   M^{-1} \l({1}/{r} \r) \r) dr \nonumber \\
	& \quad+ \kappa_1  \int_{\mu(t)}^{\infty} \widetilde A \l(\frac{\|f\|_{L^{M,\infty}(\Om)}}{\kappa_6\eps} r^{-1/{\alpha'}} \mu(t)  M^{-1} \l({1}/{\mu(t)}\r) \r) dr + 2 \intot A_\alpha(\sigma|u|) \; dx. \nonumber
	\end{align}
From equations \eqref{dir-1}, \eqref{dir-1.1} and \eqref{dir-2.7new} we deduce that
	\begin{align}
	\label{dir-2.7}
 \Big(1-\eps \kappa_1&-\frac{1}{k}\Big)\intot A(|\nabla u|) \,dx \\
	& \le   \kappa_1\int_0^{\mu(t)} \widetilde A \l(\frac{\|f\|_{L^{M,\infty}(\Om)}}{\kappa_6\eps} r^{ 1/\alpha}  M^{-1} \l({1}/{r} \r) \r) dr \nonumber \\
& \quad+ \kappa_1\int_{\mu(t)}^{\infty} \widetilde A \l(\frac{\|f\|_{L^{M,\infty}(\Om)}}{\kappa_6\eps} r^{-1/{\alpha'}} \mu(t)  M^{-1} \l({1}/{\mu(t)}\r) \r) dr+ 3 \int_{\Om_t} A_\alpha (\sigma|u|) \;dx. \nonumber
	\end{align}
Set $\widetilde k=\frac{k}{k-k\eps \kappa_1-1}$, and choose  $\eps$  so small that  $k-k\eps \kappa_1-1>0$, whence $\widetilde k>1$.  Inequality \eqref{dir-2.7} entails that
\begin{align}
	\label{dir-3}
	 \intot A(|\nabla u|) \,dx &
	\le  
	c \int_0^{\mu(t)} \widetilde A \l(\frac{\|f\|_{L^{M,\infty}(\Om)}}{\kappa_6\eps} r^{ 1/\alpha}  M^{-1} \l({1}/{r} \r) \r) dr \\
	& \qquad +
	c \int_{\mu(t)}^{\infty} \widetilde A \l(\frac{\|f\|_{L^{M,\infty}(\Om)}}{\kappa_6\eps}r^{-1/{\alpha'}} \mu(t)  M^{-1} \l({1}/{\mu(t)}\r) \r) dr
	 + \int_{\Om_t} A_\alpha (3\widetilde k \sigma|u|)\;dx \nonumber
	\end{align}
for some constant $c$.
Our next task is to 
 estimate the last integral on the right-hand side of \eqref{dir-3}.  
%
 %
Inequality \eqref{eq:poincarealpha}, applied to the function $(|u|-t)_+$, yields
  \begin{equation}\label{june33}
    \intot  A_\alpha\l(\frac{\kappa_5(|u|-t)}{(\intot A(|\nabla u|)\;dy)^{1/\alpha}}\r)\;dx\leq \intot A(|\nabla u|) \;dx.
\end{equation}
Here, we have again made use of the fact that 
 ${\rm med}((|u|-t)_+)=0$. Since 
$\lim_{t \to \infty} \intot A(|\nabla u|)\;dy =0$, there exists $t_1>\max\{t_0,   {\rm med}(u)\}$ such that 
$$6\widetilde k \sigma \leq   \kappa_5 \bigg(\intot A(|\nabla u|)\;dy\bigg)^{-1/\alpha} \quad \text{for $t \geq t_1$.}$$
Hence, owing to inequality \eqref{june33},
    \begin{align}
    \label{intAn}
 \intot A_\alpha(3\widetilde k\sigma|u|) \, dx &\le \frac{1}{2} \intot  A_\alpha\bigl(6\widetilde k \sigma (|u|-t)\bigr) \, dx+ \frac{1}{2} \intot A_\alpha(6\widetilde k \sigma t) \, dx \\
 &\le \frac{1}{2} \intot A(|\nabla u|) \;dx+\frac{1}{2}A_\alpha(6\widetilde k \sigma t)\mu(t) \quad \text{for $t \geq t_1$.} \nonumber
\end{align} 
Coupling inequality \eqref{dir-3}  with \eqref{intAn} results in
	\begin{align}
	\label{dir-4}
	\intot A(|\nabla u|) \,dx 
	& \le c \Biggl\{ \mu(t) A_\alpha (ct)+ \int_0^{\mu(t)} \widetilde A\l(\frac{\|f\|_{L^{M,\infty}(\Om)}}{\kappa_6\eps} r^{ 1/\alpha} M^{-1}\l({1}/{r}\r)  \r) dr \\
	& \qquad \qquad \qquad+ \int_{\mu(t)}^{\infty} \widetilde A \l(\frac{\|f\|_{L^{M,\infty}(\Om)}}{\kappa_6\eps}r^{- 1/{\alpha'}} \mu(t) M^{-1}\l({1}/{\mu(t)}\r) \r) dr \Biggr\} \quad \text{for $t \geq t_1$,} \nonumber
	\end{align}
for some constant $c$.
By Jensen's inequality, the coarea formula, and the relative isoperimetric inequality \eqref{isop}
	\begin{align}
	\label{dir-4.2}
	\frac{1}{\mu(t)} \intot A(|\nabla u|) \,dx & \ge A \l(\frac{1}{\mu(t)} \intot |\nabla u| \,dx\r) = A \l(\frac{1}{\mu(t)} \int_t^{\infty} \mathcal H^{n-1}(\partial ^M\{|u|> \tau\}\cap \Omega) \, d\tau \r) \\
	& \ge A \l(\frac{\kappa_4}{\mu(t)} \int_t^{\infty} \mu(\tau)^{ 1/{\alpha'}} \,d\tau \r) \quad \quad \text{for  $t \geq t_1$.} \nonumber
	\end{align}
Therefore, on defining the function $\Psi: [0,\infty) \to [0, \infty)$ as the function $\Psi_\lambda$ given by \eqref{N},   with 
\begin{equation}\label{lambda}
\lambda = \frac{ \|f\|_{L^{M, \infty}(\Om)}}{\kappa_6\eps},
\end{equation}
 and exploiting  inequality \eqref{dir-4.2}, we deduce from \eqref{dir-4} that
	\begin{equation}
	\label{dir-4.1}
	\int_t^{\infty} \mu(\tau)^{ 1/{\alpha'}} \;d\tau \le  \frac{\mu(t)}{\kappa_4} A^{-1}\big( c \l( A_\alpha(ct)+ \Psi(\mu(t)) \r)\big) \quad \text{for $t \geq t_1$.}
	\end{equation}
Now set
	\begin{equation*}
	\begin{split}
	U = \l\{t \ge t_1 : \, \Psi(\mu(t)) \ge A_\alpha(ct)\r\}. \\
	\end{split}
	\end{equation*}
We claim that
	\begin{equation}
	\label{dir-claim1}
	|U|< \infty.
	\end{equation}
To verify this claim, let us define the function $\eta: (0, \infty) \to [0, \infty)$ as 
	\begin{equation}
	\label{y}
	\eta(t)= \int_t^{\infty} \mu(\tau)^{ 1/{\alpha'}} \;d\tau \quad \text{for $t>0$.}
	\end{equation}
This function is locally absolutely  continuous in $(0, \infty)$. Moreover,   $\lim_{t \to \infty} \eta(t)= 0$ and  $\eta'(t)= -\mu(t)^{1/\alpha'}$ for a.e. $t>0$. 
Thereby,
equation \eqref{dir-4.1} implies that
	\begin{equation}
	\label{dir-4.4}
	\eta(t) \le  \frac{(-\eta'(t))^{\alpha'}}{\kappa_4}  A^{-1} \l(c \big(A_\alpha(c t)+ \Psi((-\eta'(t))^{\alpha'})\big) \r)  \quad \text{for a.e. $t>0$.}
	\end{equation}
By inequalities \eqref{dir-4.4} and   \eqref{Ak}, there exists a constant $c$ such that
    \begin{align}
    \label{dis-dir-M}
    \eta(t)  
 \le c (-\eta'(t))^{\alpha'} A^{-1}  \l(\Psi((-\eta'(t))^{\alpha'})\r) \quad \text{for $t\in U$.}
	\end{align}
Observe that, by choosing $\lambda$ as in \eqref{lambda}
in the definition of the function $\Phi_\lambda$ in \eqref{def-M} and setting, for simplicity, $\Phi=\Phi_\lambda$,  inequality \eqref{dis-dir-M} reads
\begin{equation}\label{june35}
 \eta(t) \le c \Phi\big((-\eta'(t))^{\alpha'}\big)  \quad \text{for $t\in U$.}
\end{equation}
By Lemma \ref{prop:M},   $\Phi$ is either  identically equal to $0$, or    strictly increasing in some interval of the form
$(0, \delta)$, with  $\delta>0$.  In the former case,
our assumption that $\esssup |u|=\infty$ immediately leads to a contradiction, inasmuch as the right-hand side of inequality \eqref{june35} vanishes if $t$ is sufficiently large, whereas its left-hand side is strictly positive for every $t>0$. 
\\ We may thus focus on the case when $\Phi$ is strictly increasing in $(0, \delta)$. From inequality \eqref{june35} one infers that  there exists $t_2\geq t_1$ such that
	\[
	1 \le -\frac{\eta'(t)}{\Phi^{-1}\l(\frac{\eta(t)}{c} \r)^{ 1/{\alpha'}}} \quad \text{if $t\in U\cap  (t_2, \infty)$.}
	\]
Thereby,
    \begin{align}
	    \label{meas-U}
	 |U| 
	&\leq (t_2- t_1) + \int_{U\cap  (t_2, \infty)} dt \le (t_2- t_1) + \int_{t_2}^\infty \frac{-\eta'(t)}{\Phi^{-1}\l(\frac{\eta(t)}{c} \r)^{{ 1/{\alpha'}}}}\;dt \\
	& \leq (t_2- t_1)  + c \int_0 \frac{dr}{\Phi^{-1}(r)^{ 1/{\alpha'}}}< \infty, \nonumber
	    \end{align}
where  the last inequality holds thanks to  Lemma \ref{prop:M}. Hence, claim \eqref{dir-claim1} follows.
\\
Next, set $W= [t_1, \infty) \setminus U$. Thanks to property \eqref{dir-claim1}, we have that $|W|= \infty$. Inequalities  \eqref{dir-4.1} and \eqref{Ak} ensure that 
	\begin{equation}
	\label{dir-4.4.1}
	\int_t^{\infty} \mu(\tau)^{ 1/{\alpha'}} d\tau \le c \mu(t) A^{-1} \l(A_\alpha( ct)\r) \quad \text{for $t \in W$,}
	\end{equation}
for some constant $c$, 
whence
	\begin{equation}
	\label{dir-4.5}
	\frac{1}{\displaystyle A^{-1}(A_\alpha(ct))^{ 1/{\alpha'}}} \le \frac{(c\mu(t))^{ 1/{\alpha'}}}{ \l(\int_t^{\infty} \mu(\tau)^{ 1/{\alpha'}} d\tau \r)^{ 1/{\alpha'}}} \quad \text{for $t\in W$.}
	\end{equation}
An integration of this  inequality over the set $W$ yields
	\begin{equation}
	\label{dir-4.6}
	\int_W \frac{dt}{\displaystyle A^{-1}(A_\alpha(ct))^{ 1/{\alpha'}}}  \le \int_W \frac{(c\mu(t))^{ 1/{\alpha'}}}{ \l(\int_t^{\infty} \mu(\tau)^{1/{\alpha'}} d\tau \r)^{1/{\alpha'}}} \, dt \le \frac {c}\alpha  \l(\int_{t_1}^{\infty} \mu(\tau)^{ 1/{\alpha'}} d\tau\r)^{ 1/\alpha}< \infty,
	\end{equation}
where the last inequality holds by  \eqref{dir-4.2}.
\\ Define the function $L: [t_1, \infty) \to [0, \infty)$ as
\begin{equation*}
	L(t)= \frac{1}{\displaystyle A^{-1}(A_\alpha(ct))^{ 1/{\alpha'}}} \quad \text{for $t\geq t_1$.}
	\end{equation*}
The function $L$ is decreasing, and hence bounded. Therefore,  combining equations \eqref{dir-claim1} and 
\eqref{dir-4.6} tells us that
	\begin{equation*}
	\label{dir-claim2}
	\int_{t_1}^{\infty} L(t)\, dt< \infty.
	\end{equation*}
Define the  set $I_s= [s, \infty) \cap U$   for  $s \in [t_1, \infty)$.  
%
Since $L$ is decreasing,
 	\[
	\int_{I_s} L(t) \dd t\leq 
 L(s) |I_s|.
	\]
On the other hand, owing to the reverse Hardy-Littlewood inequality 	\eqref{HLrev}, for every $l>s$ one has that
	\begin{align*}
	\int_{[s, \infty) \cap W} L(t) \dd t \geq&	\int_{s}^\infty L(t) \chi_{[s, l) \cap W}(t)\dd t \geq\int_{s}^l L(t) (\chi_{[s, l) \cap W})_*(t)\dd t \\ \nonumber & =\int_{s+ |[s,l)\cap U|}^l L(t) \dd t \geq  \int_{s+ |I_s|}^{l+|I_s|-  |[s,l)\cap U|} L(t) \dd t.
%
%
%
 	\end{align*}
Hence, if $l \geq 2|I_s|$, 
%
%
	\begin{align*}
	  \int_{[s, \infty) \cap W} L(t) \dd t  \geq  \int_{s+ |I_s|}^{s+ 2|I_s|} L(t) \dd t \ge L(s+ 2|I_s|) |I_s|.
	\end{align*}
Since $\lim_{s \to \infty} |I_s|= 0$, given $\bar k>1$, there exists   $s_1>t_1$ such that $2|I_s| \le (\bar k-1)s $ if $s \ge s_1$. Therefore, $s+ 2|I_s| \le \bar k s$, whence $L(s+ 2|I_s|) \ge L(\bar k s)$ if $s \ge s_1$. It follows that
	\begin{equation}
	\label{dir-4.9}
	\int_{[s, \infty) \cap W} L(t) \dd t \ge L(\bar k s) |I_s| \ge \int_{[\bar k s, \infty) \cap U} L(t) \dd t= \int_{I_{\bar k s}} L(t) \dd t \quad \text{for $s \ge s_1$.}
	\end{equation}
Combining inequalities \eqref{dir-4.5} and \eqref{dir-4.9} yields
	\begin{align}
	\label{dir-4.10}
\int_{\bar k s}^{\infty} \frac{dt }{\displaystyle A^{-1}(A_\alpha(ct))^{ 1/{\alpha'}}}  &=
\int_{\bar k s}^{\infty} L(t) \,dt = \int_{[\bar k s, \infty) \cap W} L(t) \,dt+ \int_{I_{\bar k s}} L(t) \,dt 
\\
& \le \int_{[\bar k s, \infty) \cap W} L(t) \dd t+ \int_{[s, \infty) \cap W} L(t) \,dt \nonumber \\
	& \le 2 \int_{[s, \infty) \cap W} L(t) \,dt   \le 2 \int_{[s, \infty) \cap W} \frac{(c\mu(t))^{ 1/{\alpha'}}}{ \l(\int_s ^{\infty} \mu(\tau)^{ 1/{\alpha'}} d\tau \r)^{ 1/{\alpha'}}} \,dt \nonumber \quad   \text{for $s\geq s_1$.}
	\end{align}
A change of variables in the integral on the leftmost side of equation \eqref{dir-4.10} enables one to deduce  that
	$$
	\int_t^{\infty} \frac{d\tau}{A^{-1}(A_\alpha(c\tau))^{ 1/{\alpha'}}}  \le c \int_{[t, \infty) \cap W} \frac{\mu(\tau)^{ 1/{\alpha'}}}{ \l(\int_t^{\infty} \mu(r)^{1/{\alpha'}} dr \r)^{1/{\alpha'}}} \,d\tau \le c \l(\int_t^\infty \mu(\tau)^ {1/\alpha'}d\tau \r)^{ 1/\alpha} \quad  \text{for $t\geq s_1$,}
	$$
for some constant $c$,
whence
    \begin{equation}
        \label{concl-1}
        \l(\int_t^{\infty} \frac{d\tau}{A^{-1}(A_\alpha(c \tau))^{1/{\alpha'}}}\r)^{\alpha} \le c \int_t^\infty \mu(\tau)^{ 1/{\alpha'}} \; d\tau \quad \text{for $t\geq s_1$.}
    \end{equation}
If
    \begin{equation*}
    \int ^{\infty} \frac{d\tau}{A^{-1}(A_\alpha(c\tau))^{1/{\alpha'}}}= \infty,
    \end{equation*}
then inequality  \eqref{concl-1} yields a contradiction. Conversely, suppose that
    \begin{equation}
        \label{concl-2}
    \int ^\infty  \frac{d\tau}{(A^{-1}(A_\alpha(c\tau)))^{1/{\alpha'}}}< \infty.
    \end{equation}
Thanks to inequality \eqref{intAn},  
for every $\lambda>0$ the function $ \omega_\lambda : [0, \infty) \to [0, \infty)$, given by
    $$
    \omega_\lambda(t)= \intot A_\alpha(\lambda |u|) \; dx \quad \text{for $t>0$,}
    $$
is actually finite-valued, and  
    \begin{equation}
        \label{concl-3}
        \lim_{t \to \infty} \omega_\lambda(t)= 0.
    \end{equation}
Since $\omega_\lambda(t) \ge A_\alpha(\lambda t) \mu(t)$  for  $t> 0$, one has that
    \begin{equation}
        \label{concl-4}
      \int_t^\infty \mu(\tau)^{1/{\alpha'}} \; d\tau \leq   \int_t^\infty \frac{\omega_\lambda(\tau)^{ 1/{\alpha'}}}{A_\alpha(\lambda \tau)^{ 1/{\alpha'}}} \; d\tau \quad \text{for $t>0$,}
    \end{equation}
for every $\lambda>0$.  Notice that, since the function $A^{-1}$ is concave and increasing, equations  \eqref{concl-2}  and \eqref{concl-3}
entail that the   integral on the right-hand side of \eqref{concl-4} converges. 
Equation \eqref{concl-3} implies, via De L'Hopital's rule, that
    \begin{equation}
        \label{concl-5}
        \lim_{t \to \infty} \int_t^\infty \frac{\omega_\lambda(\tau)^{\alpha}}{A_\alpha(\lambda \tau)^\alpha} \; d\tau \l(\int_t^\infty \frac{d\tau}{A_\alpha(\lambda \tau)^\alpha}  \r)^{-1}= 0
    \end{equation}
for every $\lambda>0$. Combining equations \eqref{concl-1}, \eqref{concl-4} and \eqref{concl-5}  tells us that
\begin{equation*}
        \lim_{t \to \infty}   \l(\int_t^{\infty} \frac{d\tau}{A^{-1}(A_\alpha(c\tau))^{1/{\alpha'}}}\r)^{\alpha} \l(\int_t^\infty \frac{d\tau}{A_\alpha(\lambda \tau)^\alpha}  \r)^{-1}= 0.
\end{equation*}
The fact that this limit holds, whatever the constant $c$ is, 
for sufficiently large  $\lambda$ depending on $c$, contradicts the inequality
\begin{equation*}
 \int_t^\infty \frac{d\tau}{A_\alpha(\lambda \tau)^\alpha} \leq c'  \l(\int_t^{\infty} \frac{d\tau}{A^{-1}(A_\alpha(c\tau))^{1/{\alpha'}}}\r)^{\alpha} \quad \text{for $t>0$,}
\end{equation*}
which holds for every $\alpha>1$ and $\lambda >2$, for a suitable constant $c'=c'(\alpha, \lambda)$, as shown in  \cite[Lemma 2]{Cianchi-boundedness}.
\\ This contradiction ensures that $u \in L^\infty(\Om)$.
\end{proof}

\medskip

\begin{proof}[Proof of Theorem \ref{thm:dir}, sketched.]
By the reasons explained at the beginning of the proof of Theorem \ref{thm:neu},  we may assume, without loss of generality, that the function $A$ satisfies condition \eqref{convint0alpha} with $\alpha =n$ and that the function $M$ satisfies the hypotheses of Lemma \ref{prop:M}, also with $\alpha =n$.
\\
Let $u$ be a weak solution to the Dirichlet problem \eqref{dirichlet} and assume by contradiction that $\displaystyle \esssup | u|= \infty$. Let $t_0>0$ be the constant  appearing in equations \eqref{ipB}--\eqref{ipD}. Fix    $t> \max\{t_0, \esssup |u_0|\}$. The fact that, in particular, $t>  \esssup |u_0|$ ensures that the function \eqref{phi} belongs to  $V_0^1 K^A(\Om)$. It can therefore be used as a test function in the weak formulation  \eqref{dir-weak-sol} of problem  \eqref{dirichlet} to deduce that
    \begin{align}
    \label{eq-dir-1}
    \intot \mathcal{A}(x, u, \nabla u) \cdot \nabla u \; dx= \intot \mathcal{B}(x, u, \nabla u) \sgn(u) (|u|- t) \; dx,
    \end{align}
where the set $\Om_t$ is defined as in \eqref{omegat}. Assumptions \eqref{hpA} and \eqref{ipB} imply that
    \begin{align}
        \label{eq-dir-2}
        \intot \mathcal{A}(x, u, \nabla u) \cdot \nabla u \; dx \ge \intot A(|\nabla u|) \; dx- \intot A_n(\sigma|u|) \; dx,
    \end{align}
whereas assumption \eqref{hpB} entails that
    	\begin{align}
	\label{eq-dir-3}
	\int_{\Om_t} \mathcal B(x, u, \nabla u) \sgn(u) (|u|- t) \; dx& \le \int_{\Om_t} f(x) (|u|- t) \; dx+ \int_{\Om_t} C(|u|) (|u|- t) \; dx \\
	& \qquad+ \int_{\Om_t} D(|u|) E(|\nabla u|) (|u|- t) \; dx. \nonumber
	\end{align}
The right-hand side of inequality \eqref{eq-dir-3} can be estimated along the same lines as the estimates established for the right-hand side of inequality \eqref{dir-2} in the proof of Theorem \ref{thm:neu}, save that Sobolev type inequalities for functions vanishing on the boundary  $\partial \Omega$ have to be exploited, instead of their counterparts for functions with vanishing median.
Specifically, from  inequality \eqref{youngcianchi} with $\alpha= n$ and inequality \eqref{intol0} we deduce that
	\begin{align}
\label{eq-dir-4}
	 \int_{\Om_t} f(x)(|u|-t) \; dx
	& \le \eps \kappa_1 \int_{\Om_t} A(|\nabla u|) \;dx+ \kappa_1 \int_0^{\mu(t)} \widetilde A \l(\frac{\|f\|_{L^{M,\infty}(\Om)}}{\kappa_3\eps} r^{ 1/n}   M^{-1} \l({1}/{r} \r) \r) dr \\
	& \quad + \kappa_1  \int_{\mu(t)}^{\infty} \widetilde A \l(\frac{\|f\|_{L^{M,\infty}(\Om)}}{\kappa_3\eps} r^{-1/{n'}} \mu(t)  M^{-1} \l({1}/{\mu(t)}\r) \r) dr. \nonumber
	\end{align}
Moreover, assumption \eqref{ipC}  implies that
    \begin{align}
        \label{eq-dir-5}
        \int_{\Om_t} C(|u|)(|u|- t) \; dx 
	 \leq  \intot A_n(\sigma|u|) \; dx,
    \end{align}
and, by assumption \eqref{ipD}, Young's inequality and property \eqref{Ak},
    \begin{align}
        \label{eq-dir-6}
        	 \int_{\Om_t} D(|u|)E(|\nabla u|) (|u|- t) \; dx \le \int_{\Om_t} \frac{1}{k} A\l(|\nabla u| \r)+  A_n(\sigma|u|) \; dx.
    \end{align}
Combining equations  \eqref{eq-dir-1}--\eqref{eq-dir-6} and arguing as in the proof of inequality \eqref{dir-3} enable one to infer that 
    \begin{align}
        \label{eq-dir-7}
         \intot A(|\nabla u|) \,dx &
	\le  c \int_0^{\mu(t)} \widetilde A \l(\frac{\|f\|_{L^{M,\infty}(\Om)}}{\kappa_3\eps} r^{1/ n}  M^{-1} \l({1}/{r} \r) \r) dr \\
	& \qquad +c \int_{\mu(t)}^{\infty} \widetilde A \l(\frac{\|f\|_{L^{M,\infty}(\Om)}}{\kappa_3\eps}r^{-1/{n'}} \mu(t)  M^{-1} \l({1}/{\mu(t)}\r) \r) dr
	 + \int_{\Om_t} A_n (3 \widetilde{k} \sigma|u|)\;dx, \nonumber
    \end{align}
for some constants $c>0$ and 
$\widetilde k >1$. 
\\ The last integral on the right-hand side of inequality  \eqref{eq-dir-7} can be bounded   as in \eqref{dir-3}--\eqref{intAn}, via inequality \eqref{eq:poincare} in the place of \eqref{eq:poincarealpha}. Eventually, we obtain that
    \begin{align}
       \label{eq-dir-8}
        \intot A(|\nabla u|) \,dx
	& \le c \Biggl\{ \mu(t) A_n (ct)+ \int_0^{\mu(t)} \widetilde A\l(\frac{\|f\|_{L^{M,\infty}(\Om)}}{\kappa_3\eps} r^{1/ n} M^{-1}\l({1}/{r}\r)  \r) dr \\
	& \qquad \qquad \qquad+ \int_{\mu(t)}^{\infty} \widetilde A \l(\frac{\|f\|_{L^{M,\infty}(\Om)}}{\kappa_3\eps}r^{-1/{n'}} \mu(t) M^{-1}\l({1}/{\mu(t)}\r) \r) dr \Biggr\} \quad \text{for $t \geq t_1$,} \nonumber
    \end{align}
for suitable constants $c>0$ and   $t_1 >  \max\{t_0, \esssup |u_0|\}$. 
\\ An analogous chain as in \eqref{dir-4.2}, with the relative isoperimetric inequality replaced by the isoperimetric inequality \eqref{isoprn} in $\RN$, yields 
\begin{align}
	\label{june40}
	\frac{1}{\mu(t)} \intot A(|\nabla u|) \,dx  \ge  A \l(\frac{n \omega_n^{\frac 1n}}{\mu(t)} \int_t^{\infty} \mu(\tau)^{ 1/{n'}} \,d\tau \r) \quad \quad \text{for  $t \geq t_1$.} 
	\end{align}
One can now start from equations \eqref{eq-dir-8} and \eqref{june40} instead of 
\eqref{dir-4} and \eqref{dir-4.2} and argue as in the proof of Theorem \ref{thm:neu} to produce a 
 contradiction to the assumption  that  $\displaystyle \esssup | u|= \infty$.
\end{proof}
	
\section{A Sobolev trace inequality in Lorentz $\Lambda$-spaces}\label{traceineq}


	
This section is devoted to a  Sobolev  trace inequality in  Lorentz $\Lambda$-spaces to be employed in the proof of Theorem \ref{thm:rob} on Robin problems. This is the content of Theorem \ref{traceLambda}. We begin with some preliminary material on general rearrangement-invariant spaces and, more specifically, on Lorentz and Marcinkiewicz spaces.

	Let $(\mathcal R, m)$ be a non-atomic, sigma-finite measure space.  A rearrangement-invariant space $X(\mathcal R)$ on $\mathcal R$ is a Banach function space, in the sense of Luxemburg, such that
\begin{equation*} 
\|u\|_{X(\mathcal R)} = \|v\|_{X(\mathcal R)}
\end{equation*}
whenever $u, v \in \M(\mathcal R)$ satisfy the equality $u^*=v^*$. 
We refer to the monograph \cite{BS} for a comprehensive treatment of the theory of rearrangement-invariant spaces. 
\\ The associate space of $X(\mathcal R)$ is denoted by $X'(\mathcal R)$. The latter is a kind of measure theoretic dual of $X(\mathcal R)$, and enters a H\"older type inequality, since  its norm is given by
$$\|u\|_{X'(\mathbb R)}= \sup_{v\in X(\Om)}\frac{\int_{\mathcal R} uv\, dm}{\|v\|_{X(\Om)}}.$$
In particular, ${(X')'(\mathcal R)}= X(\mathcal R)$.
\\ The representation space $\overline X(\mathcal R)$ of $X(\mathcal R)$ is a rearrangement-invariant space on the interval $(0, m(\mathcal R))$, endowed with the Lebesgue measure. It has the property that
\begin{equation*} 
\|u\|_{X(\mathcal R)} = \|u^*\|_{\overline X(0, m(\mathcal R))} 
\end{equation*}
for every function $u \in X(\mathcal R)$.

 Lebesgue, Orlicz, Lorentz and Orlicz-Lorentz spaces are special instances of rearrangement-invariant spaces. The Lorentz and the Marcinkiewicz rearrangement-invariant spaces are built upon  a general  quasiconcave function. 
In the present setting, a function $\varphi \colon [0, \infty) \to [0, \infty)$ is called   quasiconcave if it is increasing, vanishes only at $0$, and the function $ \frac{\varphi(s)}{s}$ is decreasing. Every quasiconcave function is equivalent, up to multiplicative constants, to a concave function \cite[pg. 71]{BS}. Moreover,
\begin{equation}\label{june45}
\min\{1, r\} \varphi (s) \leq \varphi(r \textcolor{black}{s}) \leq \max\{1, r\} \varphi (s)\quad \text{for $r,s\geq 0$.}
\end{equation}
If $\varphi$ is a quasiconcave function, then the function  $\widehat \varphi$, given by
\begin{equation}\label{hatphi}
\widehat \varphi (s) = \begin{cases}\displaystyle \frac s{\varphi (s)} \quad & \text{if $s>0$}
\\ 0  \quad & \text{if $s=0$,} 
\end{cases}
\end{equation}
 is also quasiconcave.
\\
The Lorentz space  $\Lambda_\varphi(\mathcal R)$ associated with a quasiconcave function $\varphi$ is the collection of all   functions $u \in \M(\mathcal R)$  making the norm
%
%
%
%
%
%
%
%
%
%
%
%
%
	\begin{equation*}
	\|u\|_{\Lambda_\varphi(\mathcal R)}= \int_0^{m(\mathcal R)} u^*(s) \,d\varphi (s)
	\end{equation*}
finite. 
\\ The Marcinkiewicz space $M_{\varphi}(\mathcal R)$ associated with $\varphi$ is the collection of those  functions $u \in \mathcal M(\mathcal R)$ such that the norm 
\begin{equation*}
	\|u\|_{M_{\varphi}(\mathcal R)}= \sup_{s\in (0, m(\mathcal R))} \varphi(s) u^{**}(s)
	\end{equation*}
is finite. 
One has that \cite{musil}
   \begin{equation}\label{associate}
    \Lambda'_\varphi(\mathcal R)= M_{\widehat \varphi}(\mathcal R) \quad \text{and} \quad M_{\varphi}'(\mathcal R)= \Lambda_{\widehat\varphi}(\mathcal R).
    \end{equation}
If $\Omega$ is an open set in $\RN$, we define the Sobolev type space built upon $X(\Om)$ as 
$$W^1X (\Omega) = \{u \in X(\Om) :\, \text{$u$ is weakly differentiable and $|\nabla u| \in X (\Omega)$}\}.$$
	
\begin{lemma}
\label{lemma:isopX}
Assume that $\Omega$ is an open set in $\RN$ such that $\Om \in \mathcal J_{1/{\alpha'}}$ for some $\alpha \geq n$.
Let $X(\Omega)$ be a rearrangement-invariant space. 
Then, there exists a positive constant $\kappa_9= \kappa_9(\Om)>0$ such that
\begin{equation}
        \label{grad-X}
      \|u- {\rm med}(u)\|_{X(\Omega)} \le \kappa_9  \|\nabla u\|_{X(\Omega)}
    \end{equation}
 for every function
$u \in W^1 X(\Om)$. 
\end{lemma}

\begin{proof} Let $u$ be as in the statement. Owing to \cite[Lemma 6.6]{CEG}, its signed decreasing rearrangement $u^{\circ}$ is locally absolutely continuous in $(0, |\Omega|)$. Moreover, 
$u^{\circ} \left(\frac{|\Om|}{2}\right)=  {\rm med}(u)$. Hence, 
   \begin{equation*}
        u^{\circ}(s)- {\rm med}(u)=  \int_s^{\frac{|\Om|}{2}} -(u^{\circ})'(r) \,dr \quad \text{for $s \in (0, |\Om|)$.}
    \end{equation*}
Let $T$ be the linear operator defined as
$$T h(s) =   \int_s^{\frac{|\Om|}{2}}h(r) \min\{r, |\Omega|-r\}^{- 1/{\alpha'}}\,dr \quad \text{for $s \in (0, |\Om|)$,}$$
for an integrable function $h:  (0, |\Om|)\to \mathbb R$. One can verify that $T$ is bounded from $L^1(0, |\Om|)$  into $L^1(0, |\Om|)$ and from $L^\infty(0, |\Om|)$  into $L^\infty(0, |\Om|)$, with norms depending only on $|\Om|$ and on $\alpha$. An interpolation theorem by Calder\'on \cite[Chapter 3, Theorem 2.12]{BS}  ensures that it is also bounded from $\overline X(0, |\Om|)$ into $\overline X(0, |\Om|)$, with the same dependence of the norm. Thus, there exists a constant  $c=c(|\Om|, \alpha)$ such that
\begin{equation}\label{apr10}
\|u^\circ- {\rm med}(u)\|_{\overline X(0, |\Om|)} \leq c \| -(u^{\circ})'(r) \min\{r, |\Omega|-r\}^{ 1/{\alpha'}}\|_{\overline X(0, |\Om|)}.
\end{equation}
On the other hand, since $\Omega \in \mathcal J_{1/{\alpha'}}$,  \cite[Lemma 4.1-(ii)]{CP} ensures that
\begin{equation}\label{apr11}
 \| -(u^{\circ})'(r) \min\{r, |\Omega|-r\}^{ 1/{\alpha'}}\|_{\overline X(0, |\Om|)} \leq c\|\nabla u\|_{X(\Om)}
\end{equation}
for some constant  $c=c(\Om)$. Inequality \eqref{grad-X} follows from \eqref{apr10} and \eqref{apr11}, owing to the fact that
$$\| u^\circ- {\rm med}(u)\|_{X(\Om)} =  \|u^\circ- {\rm med}(u)\|_{\overline X(0, |\Om|)}.$$
\end{proof}

Assume that $\Omega$ is an admissible domain in $\RN$, according to the definition given in Section \ref{spaces}, and let  $X(\Omega)$ be a rearrangement invariant space. Let $Y(\partial\Om)$ be the rearrangement-invariant space on $\partial \Om$ endowed with the norm obeying
\begin{align}\label{june46}
  \|u\|_{ Y'(\partial \Om)} =\|u^{**}_{\mathcal H^{n-1}}(\gamma_\Om s^{{1}/{n'}})\|_{\overline X'(0,|\Om|)}
\end{align}
for a  function $u \in \M(\partial \Omega)$, where $\gamma_\Om = \frac{\mathcal H^{n-1}(\partial \Om)}{|\Om|^{1/n'}}$. Here, the notation $u^{**}_{\mathcal H^{n-1}}$ means that  the rearrangement of $u$ is taken with respect to the measure $\mathcal H^{n-1}$ on $\partial \Om$.
Then 
 \cite[eq. (1.8)]{CKP} tells us that
\begin{equation}\label{opttr}
{\rm Tr}: W^1X(\Omega) \to Y(\partial \Om),
\end{equation}
and that $Y(\partial \Om)$ is optimal among all rearrangement-invariant target spaces in this embedding.  
Let us point out that, although the results of \cite{CKP} are stated for Lipschitz domains, they  still hold  in any admissible domain $\Omega$, because the interpolation argument exploited in their proofs still applies.
\\ Since any { admissible domain} belongs to the class $\mathcal J_{1/{n'}}$, 
combining embedding \eqref{opttr} with Lemma \ref{lemma:isopX} entails that there exists a constant $\kappa_{10}$ such that
%
    \begin{align*} 
        \|{\rm Tr} \, u - {\rm med}(u)\|_{Y(\rand)} \le \kappa_{10} \|\nabla u \|_{X(\Om)}
    \end{align*}
for every $u \in W^1X(\Om)$. 

The general embedding \eqref{opttr} can be implemented in the special case when $X$ is a Lorentz space of the type introduced above. 
  Given a non-increasing function  $\zeta : [0, \infty) \to [0, \infty)$, we denote  by
$\vartheta : [0, \infty) \to [0, \infty)$  the function defined as
\begin{equation}
	\label{vartheta}
	\vartheta(r)= r^{{1}/{n}} \int_0^{r^{1/n'}} \zeta(\rho) \;d\rho \quad \text{for $r\geq 0$,}
	\end{equation}
and by $\varpi : [0, \infty) \to [0, \infty)$ the function  defined  as
\begin{equation}
	\label{w}
 \varpi(r)=\int _0^r\zeta (\rho) \;d\rho \quad \text{for $r\geq 0$.}
\end{equation}
The monotonicity   of the function $\zeta$ ensures that both these functions are quasiconcave.

\begin{theorem}\label{traceLambda}
Assume that $\Omega$ is {an admissible domain} in $\RN$. Let  $\zeta : [0, \infty) \to [0, \infty)$ be a non-increasing function and let $\vartheta$ and $\varpi$ be the functions associated with $\zeta$ as in \eqref{vartheta} and \eqref{w}. Then,
\begin{equation}\label{traceLambda1}
{\rm Tr}: W^1\Lambda _\vartheta (\Omega)  \to \Lambda _\varpi (\partial \Omega).
\end{equation}
Moreover, there exists a constant $\kappa_{11}$ such that
\begin{equation}
	\label{dise-boundary}
	\|{\rm Tr} \, u - {\rm med}(u)\|_{\Lambda _\varpi (\partial \Om)}  \le \kappa_{11}\|\nabla u\|_{\Lambda_\vartheta (\Om)} 
	\end{equation}
for every $u \in W^1\Lambda_\vartheta (\Omega)$.
\end{theorem}

The proof of Theorem  \ref{traceLambda} requires the following lemma.

\begin{lemma}\label{equivmarc}
Let $\varphi$ be any quasiconcave function.  Assume that   $h \in \mathcal M(0, \infty)$ and let $\gamma, \ell >0$. Then 
\begin{equation}\label{equivmarc1}
\sup_{s\in (0, \gamma \ell^{1/n'})}h^{**}(s)\varphi(\gamma ^{-n'}s^{n'}) \leq \sup_{s\in (0, \ell)} \frac{\varphi (s) }s\int_0^s h^{**}(\gamma r^{\frac 1{n'}})\, dr\,  \leq n\sup_{s\in (0, \gamma \ell^{1/n'})}h^{**}(s)\varphi(\gamma ^{-n'}s^{n'}).
\end{equation}
\end{lemma}
\begin{proof} After rescaling, we may assume that $\gamma =1$. Equation \eqref{equivmarc1} will thus follow if we show that
	\begin{equation}\label{chain2}
	  h^{**}(s^{1/n'}) \le \frac{1}{s} \int_0^s h^{**}(r^{1/n'}) \;dr \le n h^{**}(s^{1/n'}) \quad \text{for $s >0$,}
	\end{equation}
for every $h \in \mathcal M(0, \infty)$.
\\ The first inequality in \eqref{chain2} holds since the function $h^{**}(s^{1/n'})$ is non-increasing.
\\ The second inequality is a consequence of the following chain:
	\begin{align*}
	\frac{1}{s} \int_0^s h^{**}(r^{1/n'}) \;dr&= \frac{1}{s} \int_0^s \frac{1}{r^{1/n'}} \biggl(\int_0^{r^{1/n'}} h^*(\rho) \;d\rho \biggr) dr = \frac{n'}{s} \int_0^{s^{1/n'}} \varsigma^{\frac{1}{n-1}-1} \l(\int_0^\varsigma h^*(\rho) \;d\rho \r) d\varsigma \\
& = \frac{n'}{s} \int_0^{s^{1/n'}} \biggl(\int_{\rho}^{s^{1/n'}} \varsigma^{\frac{1}{n-1}-1} \;d\varsigma \biggr) h^*(\rho) \;d\rho  = \frac{n}{s} \int_0^{s^{1/n'}} \l(s^{1/n}- \rho^{1/(n-1)} \r) h^*(\rho) \;d\rho \nonumber \\
	& \le \frac{n}{s} \int_0^{s^{1/n'}} s^{1/n} h^*(\rho) \;d\rho  = n s^{1/n-1} \int_0^{s^{1/n'}} h^*(\rho) \;d\rho  = n h^{**}(s^{1/n'}) \quad \text{for $s >0$.} \nonumber
	\end{align*}
\end{proof}

\begin{proof}[Proof of Theorem \ref{traceLambda}]
Let $Y(\partial\Omega)$ be the optimal rearrangement-invariant target space in the trace embedding \eqref{opttr} corresponding to the domain space $X(\Om) = \Lambda _\vartheta (\Om)$.
Let $\widehat \vartheta$ be the quasiconcave function associated with $\vartheta$ as in \eqref{hatphi}.  The following chain holds:
\begin{align}\label{june47}
\|{\rm Tr} u\|_{Y'(\partial \Om)} &= \|({\rm Tr}\, u)^{**}_{\mathcal H^{n-1}}(\gamma _\Om s^{1/{n'}})\|_{\Lambda '_\vartheta (0, |\Om|)} = \|({\rm Tr}\, u)^{**}_{\mathcal H^{n-1}}(\gamma _\Om s^{1/{n'}})\|_{M_{\widehat \vartheta}(0, |\Om|)}  
\\ \nonumber &  =\sup_{s\in (0, |\Om|)} \frac{\widehat \vartheta (s)}s\int_0^s({\rm Tr}\, u)^{**}_{\mathcal H^{n-1}}(\gamma_\Om r^{\frac 1{n'}})\, dr
   \simeq  \sup _{s\in (0, \mathcal H^{n-1}(\partial \Om))}({\rm Tr}\, u)^{**}_{\mathcal H^{n-1}}(s) \, \widehat \vartheta (\gamma_\Om^{-n'}s^{n'})
\\  \nonumber & \simeq\sup _{s\in (0, \mathcal H^{n-1}(\partial \Om))}({\rm Tr}\, u)^{**}_{\mathcal H^{n-1}}(s) \, \widehat \vartheta (s^{n'})= \|{\rm Tr}\, u \|_{M_{\widehat \vartheta (s^{n'})}(\partial \Omega)}=\|{\rm Tr}\, u\|_{M_{\frac{s}{\varpi(s)}}(\partial \Omega)}.
\end{align}
Here, the first equality holds by equation \eqref{june46}, the second one by equation \eqref{associate}, 
the first equivalence by Lemma \ref{equivmarc}, and the second equivalence by equation \eqref{june45}, where equivalence is understood up to multiplicative constants independent of $u$.  Equation \eqref{june47} tells us that
\begin{equation*} 
Y'(\partial \Om) = M_{\frac{s}{\varpi(s)}}(\partial \Omega),
\end{equation*}
up to equivalent norms. Hence, thanks to property \eqref{associate},
\begin{equation*} 
Y(\partial \Om) = M_{\frac{s}{\varpi(s)}}'(\partial \Omega) = \Lambda _{\varpi(s)}(\partial \Omega),
\end{equation*}
up to equivalent norms. Embedding \eqref{traceLambda1} is thus established. 
%
\\ Inequality \eqref{dise-boundary} is a consequence of embedding \eqref{traceLambda1} and of inequality \eqref{grad-X}.
\end{proof}

Theorem \ref{traceLambda} has the following corollary. 
\begin{corollary}\label{cortrace}
Under the same assumptions as in Theorem \ref{traceLambda}, there exists a constant $\kappa_{12}$ such that
\begin{equation}\label{cor1}
\int_0^{\mathcal H^{n-1}(\partial \Om)} ({\rm Tr}\, u - {\rm med}(u))_{\mathcal H^{n-1}}^*(r) \, \zeta (r) \; dr \le \kappa_{12} \int_0^{|\Om|}|\nabla u|^*(r)\, r^{-1/n'}\int_0^{r^{1/n'}} \zeta (\rho) \;d\rho \, dr
\end{equation}
 for every  $u \in W^1\Lambda_\vartheta (\Omega)$.
\end{corollary}

Corollary \ref{cortrace}   is a straightforward consequence of  inequality \eqref{dise-boundary}  and of the next lemma.

\begin{lemma}
 Let  $\zeta : [0, \infty) \to [0, \infty)$ be a non-increasing function, and let $\vartheta$   be the function associated with $\zeta$ as in \eqref{vartheta}. 
Then,
%
%
	\begin{equation}
	\label{chain}
\frac 1n \vartheta (r) \leq r\, \vartheta'(r) \le \vartheta (r)  \quad \text{for $r>0$.}
%
	\end{equation}
\end{lemma}

\begin{proof}
Since
	\begin{equation}\label{derivata}
	\begin{split}
	\vartheta'(r)  
	= \frac{1}{n} r^{-1/n'} \int_0^{r^{1/n'}} \zeta (\rho) \;d\rho+ \frac{1}{n} \zeta(r^{1/n'}) 
\quad \text{for $r>0$,}
	\end{split}
	\end{equation}
the first inequality in \eqref{chain} holds since $\zeta$ is nonnegative.
The second inequality also follows from  equation \eqref{derivata}, inasmuch as, thanks to  the monotonicity of the function $\zeta$, 
	\[
	\begin{split}
	\vartheta'(r)
	 \le \frac{1}{n} r^{-1/n'} \int_0^{r^{1/n'}}\zeta (\rho) \;d\rho+ \frac{1}{n'} r^{-1/n'}  \int_0^{r^{1/n'}} \zeta (\rho) \;d\rho 
	= \frac{1}{r^{1/n'}} \int_0^{r^{1/n'}} \zeta (\rho) \;d\rho \quad \text{for $r>0$.}
	\end{split}
	\]
\end{proof}

\section{Proof of Theorem \ref{thm:rob}} 

The proof of Theorem \ref{thm:rob} requires a few additional   technical lemmas. Lemmas \ref{lemmaAT}--\ref{lemma:AT2} concerns properties of the Young function $A_T$ appearing in the Orlicz-Sobolev trace embedding \eqref{trace-ineq1} and in inequality \eqref{traceint}.

\begin{lemma}\label{lemmaAT} 
Let   $A$ be a Young function  fulfilling condition \eqref{intconv0}
and let $A_T$ be the function defined by equation \eqref{AT}. Then,
\begin{equation}
	\label{AT2}
	A_T (t) \le \frac{A(H^{-1}(t))}{H^{-1}(t)}\, t \quad \text{for $t >0$.}
	\end{equation}
\end{lemma}

\begin{proof}
Inequality \eqref{AT2} is a straightforward consequence of the monotonicity of the function $\frac{A(H^{-1}(t))}{H^{-1}(t)}$, which in turn follows from property \eqref{mono}.
\end{proof}

\begin{lemma}\label{lemma:ATdelta2}
Let   $A$ be a Young function satisfying conditions  \eqref{intconv0} and \eqref{intdiv}.
Let $A_T$ be the function defined by equation \eqref{AT}. 
If $A \in \nabla_2$ globally, then $A_T \in \nabla_2$ globally. 
\end{lemma}

\begin{proof}
Since $A \in \nabla_2$ globally,  there exists
$\eps>0$ such that
%
%
the function
	\begin{equation}
	\label{quoz}
	 \frac{A(t)}{t^{1+\eps}} \text{ is increasing,}
	\end{equation}
see \cite[Theorem 3, Chapter 2]{RR}.
In order to show that $A_T \in \nabla_2$ globally it suffices to show that
	\begin{equation}
	\label{claim3}
	\inf_{t>0} \frac{t A'_T(t)}{A_T(t)}> 1,
	\end{equation}
see \cite[Theorem 3, Chapter 2]{RR} again.
One has that
	\begin{align}
	\label{a1}
	\inf_{t>0} \frac{t A_T'(t)}{A_T(t)}&= \inf_{t>0} \frac{ t \frac{A(H_n^{-1}(t))}{H_n^{-1}(t)}}{ \int_0^t \frac{A(H_n^{-1}(\theta))}{H_n^{-1}(\theta)} \,d\theta}= \inf_{\tau >0} \frac{ \frac{A(\tau) H_n(\tau)}{\tau}}{  \int_0^{H(\tau)} \frac{A(H_n^{-1}(\theta))}{H_n^{-1}(\theta)} \,d\theta} \\
	&=  \inf_{\tau >0}n' \frac{\frac{A(\tau) H(\tau)}{\tau}}{ \int_0^{\tau} \big(\frac{A(\upsilon)}{\upsilon}\big)^{\frac{n-2}{n-1}} H_n(\upsilon)^{-\frac{1}{n-1}}  \,d\upsilon}, \nonumber
	\end{align}
where the last equality follows  by a change of variable in the integral.
Thanks to property \eqref{quoz} and to an integration by parts, one obtains that
	\begin{align}
	\label{a2}
	\int_0^t \l(\frac{A(\tau)}{\tau}\r)^{\frac{n-2}{n-1}} H_n(\tau)^{-\frac{1}{n-1}} \,d\tau &= 
%
\int_0^t \frac{A(\tau)}{\tau^{1+\eps}} \frac{d}{d\tau} \l(H_n(\tau)^{n'} \r) \tau^{\eps} H_n(\tau)^{-\frac{1}{n-1}} \,d\tau \\
	& \le \frac{A(t)}{t^{1+\eps}} \int_0^t  \frac{d}{d\tau}\l(H_n(\tau)^{n'}\r) H_n(\tau)^{-\frac{1}{n-1}} \,\tau^{\eps}  \,d\tau  \nonumber \\
	&= n'\frac{A(t)}{t^{1+\eps}} \int_0^t \tau^\eps H_n'(\tau) \; d\tau \nonumber \\
	& = n'\frac{A(t)}{t^{1+\eps}} \l(H_n(t) t^{\eps}- \eps \int_0^t H_n(\tau) \tau^{\eps-1} \,d\tau \r) \quad \text{for $t >0$.} \nonumber
	\end{align}
We claim that, if $\eps < \frac 1n$, then
	\begin{equation}
	\label{claim4}	\text{the function }  H_n(\tau) \tau^{\eps-1} \text{ is decreasing}.
	\end{equation}
Of course, it suffices to  prove  that the function $ \l(H_n(\tau) \tau^{\eps- 1}\r)^{n'}$ satisfies the same property, which  follows from the following chain:
	\begin{equation*}
	\begin{split}
	\frac{d}{d\tau} \l(\l(H_n(\tau) \tau^{\eps- 1}\r)^{n'}\r)&
= \frac{d}{d\tau} \l(\frac{1}{\tau^{n'(1-\eps)}} \int_0^\tau \l(\frac{\theta}{A(\theta)}\r)^{\frac{1}{n-1}} d\theta \r) 
	\\ &= \frac{1}{\tau^{n'(1-\eps)+ 1}} \l[\l(\frac{\tau}{A(\tau)}\r)^{\frac{1}{n-1}} \tau- n'(1-\eps) \int_0^\tau \l(\frac{\theta}{A(\theta)}\r)^{\frac{1}{n-1}} d\theta \r] \\
	& \le \frac{ 1- n'(1-\eps)}{\tau^{n'(1-\eps)+ 1}} \l(\frac{\tau}{A(\tau)}\r)^{\frac{1}{n-1}} \tau   < 0 \quad \text{for $\tau>0$.}
	\end{split}
	\end{equation*}
Observe that the last but one inequality holds thanks to property \eqref{mono}.
Property \eqref{claim4} ensures that
	\begin{equation}\label{apr3}
	\int_0^t H_n(\tau) \tau^{\eps-1} d\tau \ge H(t) t^{\eps} \quad \text{for $t >0$.}
	\end{equation}
Inequalities \eqref{a2} and \eqref{apr3} imply that
	\[
	\begin{split}
	\int_0^t \l(\frac{A(\tau)}{\tau}\r)^{\frac{n-2}{n-1}} H_n(\tau)^{-\frac{1}{n-1}} \,d\tau  \le n'\frac{A(t)}{t^{1+\eps}} \l(H_n(t) t^{\eps}-  \eps H_n(t) t^{\eps} \r) 
	=n' \frac{A(t)}{t} H_n(t) (1-\eps) \quad \text{for $t >0$,}
	\end{split}
	\]
whence 
	\[
	n' \frac{  \frac{A(t)}{t} H_n(t)}{ \int_0^t \big(\frac{A(\tau)}{\tau}\big)^{\frac{n-2}{n-1}} H_n(\tau)^{-\frac{1}{n-1}} \,d\tau} \ge \frac{1}{1-\eps} \quad \text{for $t>0$.}
	\]
This in turn implies that
	\[
	\inf_{t>0} n' \frac{  \frac{A(t)}{t} H_n(t)}{ \int_0^t \l(\frac{A(\tau)}{\tau}\r)^{\frac{n-2}{n-1}} H_n(\tau)^{-\frac{1}{n-1}} \,d\tau}\geq  \frac{1}{1-\eps} > 1.
	\]
Owing to equation \eqref{a1},  inequality  \eqref{claim3} hence follows.
\end{proof}

\begin{lemma}
\label{lemma:AT2}
Under the same assumption as in Lemma \ref{lemma:ATdelta2},
\begin{equation}\label{apr4}
	\lim_{\tau \to 0^+} \l(\sup_{t>0} \frac{A_T(\tau t)}{\tau A_T(t)} \r)= 0.
	\end{equation}
\end{lemma}

\begin{proof}
By Lemma  \ref{lemma:ATdelta2} we have that
$A_T \in \nabla_2$ globally. Hence,   there exists $\eps>0$ such that the function $\frac{A_T(t)}{t^{1+\eps}}$ is increasing. Therefore,
	\[
	\frac{A_T(\tau t)}{(\tau t)^{1+\eps}} \le \frac{A_T(t)}{t^{1+\eps}} \quad \text{ if $t>0$ and  $\tau \in (0, 1)$,}
	\]
whence
	\[
	\sup_{t>0} \frac{A_T(\tau t)}{\tau A_T(t)} \le \tau^{\eps} \quad \text{for $\tau \in (0,1)$.}
	\]
Passing to the limit as $\tau \to 0^+ $ yields  \eqref{apr4}.
\end{proof}

The next lemma is concerned with  properties of an equivalent form of the integrand in condition \eqref{int G}.
Specifically,
let $A$  and $N$ be Young functions fulfilling  \eqref{int G}. A change of variables in the integral appearing in this condition shows that it is equivalent to 
	\begin{equation}
	\label{int G old}
	\int _0 A^{-1} \l(\frac{1}{s} \int_{s^{-1/n'}}^{\infty} \widetilde A\l( r  \int_0^{1/r} N^{-1}(1/\rho) \,d\rho\r)\frac{dr}{r^{n'+1}} \r)\, \frac {ds}{s^{1/n'}}<\infty\,.
	\end{equation}
Define the function $ \Theta \colon [0, \infty) \to [0, \infty)$ as
	\begin{equation}
	\label{def-G}
	\Theta(s)= s A^{-1} \l(\frac{1}{s} \int_{s^{-1/n'}}^{\infty}  \widetilde A\l(r\int_0^{1/r} N^{-1}(1/\rho) \,d\rho\r)\frac{dr}{r^{n'+1}}  \r) \quad \text{for $s >0$.}
	\end{equation}

\begin{lemma}
\label{prop:G}
Let $A$  be Young function such that $A\in \nabla _2$ globally and let $ N$ be a Young function  satisfying condition \eqref{int G}. 
Let $\Theta$ be the function given by \eqref{def-G}. Then,
    \begin{itemize}
   \item[(i)] $\lim _{s\to 0^+}\Theta(s)=0$;

    \item[(ii)] $\Theta$ is strictly increasing;
        
    \item[(iii)] Condition \eqref{int G} is equivalent to 
\begin{equation*}
\int_0\frac{ dr}{\Theta ^{-1}(r)^{1/n'}} \;<\infty.
\end{equation*}
    \end{itemize}
    

\end{lemma}

\begin{proof} Property (i) follows from property (ii) and  assumption \eqref{int G}, in the equivalent form of \eqref{int G old}.
\\ In order to establish property (ii), 
 we begin by showing that
	\begin{equation}
	\label{G1}
	\text{the function } \  \frac{1}{s} \int_{s^{-{1}/{n'}}}^{\infty} \frac{\widetilde A\l( r \int_0^{1/r} N^{-1}(1/\rho) \,d\rho\r)}{r^{n'+1}} \,dr \  \text{ is decreasing}.
	\end{equation}
Indeed, 
	\begin{align*}
&	\frac{d}{ds} \l[\frac{1}{s} \int_{s^{-1/n'}}^{\infty} \frac{\widetilde A\l( r \int_0^{\frac{1}{r}} N^{-1}(1/\rho) \,d\rho\r)}{r^{n'+1}} \,dr \r]  \\
	 &= -\frac{1}{s}  \l[\frac{1}{s} \int_{s^{-1/n'}}^{\infty} \frac{\widetilde A\l( r \int_0^{1/r} N^{-1}(1/\rho) \,d\rho\r)}{r^{n'+1}} \,dr- \frac{1}{n'} \widetilde A\l( s^{-1/n'} \int_0^{s^{1/n'}} N^{-1}(1/\rho) \,d\rho\r) \r] \\
 &	 \le -\frac{1}{s} \l[\frac{1}{s} \widetilde A \l(\ s^{-1/n'} \int_0^{s^{1/n'}} N^{-1} (1/\rho) \,d\rho \r) \int_{s^{-1/n'}}^{\infty} \frac{dr}{r^{n'+1}} - \frac{1}{n'} \widetilde A \l( s^{-1/n'}  \int_0^{s^{1/n'}} N^{-1}(1/\rho) \,d\rho \r) \r] \\
 &= 0.
\end{align*}
In particular, the inequality in the chain above holds since, thanks to the monotonicity of the functions $\widetilde A$ and $ N^{-1}$, the function 
$$ 
\widetilde A\l( r \int_0^{\frac{1}{r}} N^{-1}(1/\rho) \,d\rho\r) \,\, \text{ is non-decreasing.}
$$
On the other hand,   equation \eqref{def-G} can be rewritten as
	\begin{equation}
	\label{G-equiv}
	\begin{split}
	 \Theta(s)
= \frac{A^{-1} \l(\frac{1}{s} \int_{s^{-1/n'}}^{\infty} \frac{\widetilde A\l( r \int_0^{1/r} N^{-1}(1/\rho) \,d\rho \r)}{r^{n'+1}} \,dr \r)}{\frac{1}{s} \int_{s^{-1/n'}}^{\infty} \frac{\widetilde A\l( r \int_0^{1/r} N^{-1}(1/\rho) \,d\rho \r)}{r^{n'+1}} \,dr} \int_{s^{-1/n'}}^{\infty} \frac{\widetilde A\l( r \int_0^{1/r} N^{-1}(1/\rho) \,d\rho \r)}{r^{n'+1}} \,dr \quad \text{for $s>0$.}
	\end{split}
	\end{equation}
The quotient on the right-hand side of equation \eqref{G-equiv} is non-decreasing, owing to properties \eqref{mono} and \eqref{G1}. Moreover, the last integral in \eqref{G-equiv}  is strictly increasing, inasmuch as, by property \eqref{deltanabla},   $\widetilde A \in \Delta_2$, and the latter property  implies that $\widetilde A(t)>0$ for $t>0$. 
\\ Property (iii) can be verified through a chain analogous to \eqref{apr20}. 
\end{proof}
\begin{proof}[Proof of Theorem \ref{thm:rob}] Without loss of generality, the same properties of the functions $A$ and $M$ as in the proof of Theorem \ref{thm:neu}, with $\alpha =n$, can be assumed to hold. This is again possible since   conditions  \eqref{ipBn}--\eqref{ipDn}, \eqref{new M}, \eqref{dis-AT} and \eqref{int G}  depend on the values of the functions involved only for large values of their argument. For the same reason, after a possible replacement  by an equivalent Young function near zero, the function $A$ can be assumed to fulfill the $\nabla_2$-condition globally, and not just near infinity.
\\ We argue by contradiction, and suppose that
$u$ is a weak solution to problem \eqref{robin} such that $\displaystyle \esssup |u|= \infty$. 
Fix any $t >\{t_0, {\rm med}(|u|)\}$, where $t_0$ is the constant appearing in conditions \eqref{ipBn}--\eqref{ipDn} and \eqref{dis-AT}. Choosing the test function \eqref{phi} in equation \eqref{weak-sol} results in 
	\begin{equation}
	\label{1}
	\int_{\Omt} \mathcal A(x, u, \nabla u) \cdot \nabla u \; dx= \int_{\Omt} \mathcal B(x, u, \nabla u) \sgn(u) (|u|-t) \;dx+ \int_{E_t} \mathcal C(x, u) \sgn(u) (|u|- t) \;d\ch^{n-1}.
	\end{equation}
Here, $\Omega_t$ is the set defined as in \eqref{omegat}, and 
    \begin{equation*}
        E_t= \{x \in \rand : \, |u(x)|> t\} .
    \end{equation*}
The integrals over the set $\Omt$ in equation \eqref{1} can be   estimated as in the proof of Theorem \ref{thm:neu}. One has just to set $\alpha =n$, inasmuch as,  being a bounded Lipschitz domain,  the set $\Omega \in \mathcal J_{ 1/{n'}}$.
So doing, inequalities \eqref{dir-1.1} and \eqref{dir-2.7new} turn into
    \begin{equation}
        \label{robin-1}
        \intot \mathcal{A}(x, u, \nabla u) \cdot \nabla u \; dx \ge \intot A(|\nabla u|) \; dx- \intot A_n(\sigma |u|)\; dx
    \end{equation}
and
\begin{align}
	\label{robin-2}
& \int_{\Om_t} \mathcal B(x, u, \nabla u) \sgn(u) (|u|- t) \; dx \\
& \le \Big(\eps \kappa_1 + \frac 1k\Big) \int_{\Om_t} A(|\nabla u|) \;dx+ \kappa_1 \int_0^{\mu(t)} \widetilde A \l(\frac{\|f\|_{L^{M,\infty}(\Om)}}{\kappa_6\eps} r^{ 1/n}   M^{-1} \l({1}/{r} \r) \r) dr \nonumber \\
	&  \qquad +\kappa_1  \int_{\mu(t)}^{\infty} \widetilde A \l(\frac{\|f\|_{L^{M,\infty}(\Om)}}{\kappa_6\eps} r^{-1/{n'}} \mu(t)  M^{-1} \l({1}/{\mu(t)}\r) \r) dr + 2 \intot A_n(\sigma|u|) \; dx. \nonumber
	\end{align}
%
%
%
In order to estimate the last integral in equation \eqref{1}, we make use of assumption \eqref{ipCn} and obtain that
	\begin{equation}
	\label{3}
	\begin{split}
	\intort \mathcal C(x, u) \sgn(u) (|u|- t) \,d\mathcal H^{n-1}
	& \le \intort g(x)(|u|- t) \,d\mathcal H^{n-1}+ \intort F(|u|) (|u|- t) \,d\mathcal H^{n-1}.
	\end{split}
	\end{equation}
By the Hardy-Littlewood inequality \eqref{HL}, the assumption that $g \in L^{N, \infty}(\rand)$, and the trace  inequality  \eqref{cor1}  applied with $\zeta (r) = N^{-1}(1/r)$ to the function $(|u|-t)_+$  one can deduce that
	\begin{align}
	\label{3.1}
	\intort g(x)(|u|- t) \,d\mathcal H^{n-1} & \le \int_0^{\mathcal H^{n-1}(\rand)} g^*(r) (|u|-t)_+^*(r) \,dr \\
	& \le \|g\|_{L^{N, \infty}(\rand)} \int_0^{\mathcal H^{n-1}(\rand)} N^{-1}(1/r) (|u|-t)_+^*(r) \,dr \nonumber \\
& \le \kappa_{12} \|g\|_{L^{N, \infty}(\rand)} \int_0^{|\Om|} (|\nabla u|\chi_{\Om_t})^*(r)\, r^{-1/n'}\int_0^{r^{1/n'}}  N^{-1}(1/\rho)  \;d\rho \, dr \nonumber.
	\end{align}
Note that the application of inequality \eqref{cor1} is legitimate, since ${\rm med}((|u|-t)_+)=0$, by the choice $t> {\rm med}(|u|)$.
\\ Next, since the function $(|\nabla u|\chi_{\Om_t})^*$ vanishes in $(\mu(t), |\Omega|)$, by   inequality \eqref{june21} we have that
	\begin{align}
	\label{3.2}
 \int_0^{|\Om|} & (|\nabla u|\chi_{\Om_t})^*(r)\, r^{-1/n'}\int_0^{r^{1/n'}}  N^{-1}(1/\rho)  \;d\rho \, dr 
\\ 
& =  \int_0^{|\Om|}  (|\nabla u|\chi_{\Om_t})^*(r)\,\chi_{[0,\mu(t)]}(r)\, r^{-1/n'}\int_0^{r^{1/n'}}  N^{-1}(1/\rho)  \;d\rho \, dr \nonumber
\\
& \leq 
	 \eps\int_0^{|\Om|} A( (|\nabla u|\chi_{\Om_t})^*)   \,dr
	+\eps  \int_0^{|\Om|} \widetilde A\l(\chi_{[0,\mu(t)]}(r) \frac{r^{-1/n'}}{\eps} \int_0^{r^{1/n'}} N^{-1}(1/\rho) \;d\rho \r) dr \nonumber
\\
&   =\eps\int_{\Om_t} A(|\nabla u|)   \,dx
	+\eps  \int_0^{\mu(t)} \widetilde A\l(\frac{r^{-1/n'}}{\eps} \int_0^{r^{1/n'}} N^{-1}(1/\rho) \;d\rho \r) dr \nonumber
\\
&   =\eps\int_{\Om_t} A(|\nabla u|)   \,dx
	+\eps  n' \int_{\mu(t)^{-1/n'}}^{\infty}  \widetilde A \l(\frac{\varsigma}{\eps} \int_0^{1/\varsigma} N^{-1}(1/\rho) \;d\rho \r)\frac{d\varsigma}{\varsigma^{n'+1}}. \nonumber
	\end{align}
Coupling equation  \eqref{3.1} with \eqref{3.2} yields
	\begin{align}
	\label{3.4}
	&\intort g(x) (|u|- t) \,d \mathcal H^{n-1} \\
	& \le \kappa_{12}\eps \|g\|_{L^{N, \infty}(\rand)} \l[\intot A(|\nabla u|) \,dx+  n'  \int_{\mu(t)^{-1/n'}}^{\infty} \widetilde A \l(\frac{r}{\eps} \int_0^{1/r} N^{-1}(1/\rho) \;d\rho \r)\frac{dr}{r^{n'+1}}\r]. \nonumber
	\end{align}
The second integral on the right-hand side of inequality  \eqref{3} can be bounded as follows. 
Thanks to assumption \eqref{dis-AT}, 
\begin{align}
	\label{3.5}
	 & \intort F(|u|) (|u|- t) \,d \mathcal H^{n-1}
	 \le \int_{E_t} \frac{A_T(\sigma|u|)}{|u|}(|u|- t) \,d \HN \\
 & \le \frac{1}{2} \intort  A_T(2\sigma (|u|- t)) \,d\HN+  \frac{1}{2} \intort \frac{A_T(2\sigma t)}t(|u|- t) \,d\HN  \nonumber \\
	& = \frac{1}{2} \intort  A_T \l(\frac{2\sigma \l(\intot A(|\nabla u|) \,dx\r)^{1/n} (|u|- t)}{ \l(\intot A(|\nabla u|) \,dx \r)^{1/n}}\r) d \HN+\frac{1}{2} \intort \frac{A_T(2\sigma t)}t(|u|- t) \,d\HN. \nonumber
	\end{align}
%
%
%
By Lemma \ref{lemma:AT2}, for every $\eps>0 $ there exists $\delta>0$ such that $A_T(\theta \tau)< \eps \theta A_T(\tau)$ for every $\tau >0$ and  every $\theta \in (0,\delta)$. Let $\kappa_8$ be the constant appearing in inequality \eqref{traceint}.  Let $t_1>t_0$ be so large that 
$$ \frac{2\sigma}{\kappa_8}\l(\intot A(|\nabla u|) \,dx\r)^{1/n} <\delta.$$
Hence  by inequality  \eqref{traceint}, if $t>t_1$, then
%
	\begin{align}\label{apr20-1}
 \intort  A_T& \l(\frac{2\sigma \l(\intot A(|\nabla u|) \,dx\r)^{1/n} (|u|- t)}{ \l(\intot A(|\nabla u|) \,dx \r)^{1/n}}\r) d \HN 
\\ 
\leq  & \frac{2 \eps \sigma}{\kappa_8}\l(\intot A(|\nabla u|) \,dx\r)^{1/n} 
 \intort  A_T \l(\frac{\kappa_8 (|u|- t)}{ \l(\intot A(|\nabla u|) \,dx \r)^{1/n}}\r) d \HN \nonumber
\\
&  \leq 
  \frac{2 \eps \sigma}{\kappa_8} \l(\intot A(|\nabla u|) \,dx\r)^{1/n} \l(\intot A(|\nabla u|) \,dx\r)^{1/n'}  =
 \frac{2 \eps \sigma}{\kappa_8}\intot A(|\nabla u|) \,dx. \nonumber
	\end{align}
Next, by inequality \eqref{traceint} with $A(t)=t$,  which corresponds to a standard trace inequality for functions in $W^{1,1}(\Om)$, there exists a constant $c$ such that
\begin{align}
	\label{3.6}
	 \intort \frac{A_T(2\sigma t)}t(|u|- t) \,d\HN & 
	 \le c \intot \frac{A_T(2\sigma t)}t |\nabla u| \,dx 
\\
	& \le c \eps \bigg(\intot \widetilde A \l( \frac{A_T(2\sigma t)}{\eps t}\r) dx+ \intot A(|\nabla u|) \,dx\bigg) \nonumber \\
	&= c\eps \mu(t)\widetilde A \l( \frac{A_T(2\sigma t)}{\eps t}\r) + c \eps \intot A(|\nabla u|) \,dx,\nonumber
	\end{align}
where the function $\mu$ is defined as in \eqref{mu}. Since $A \in \nabla_2$,
property \eqref{deltanabla} ensures  that $\widetilde A \in \Delta_2$. Hence
	\begin{align}
	\label{3.7}
\widetilde A \l( \frac{A_T(2\sigma t)}{\eps t}\r)&  \le \widetilde A \l(\frac {2\sigma}\eps \frac{A(H^{-1}(2\sigma t))}{H^{-1}(2\sigma t)}\r)   \le c \widetilde A\l(\frac{A(H^{-1}(2\sigma t))}{H^{-1}(2\sigma t)}\r) \\ 
& \le c A(H^{-1}(2\sigma t))   = c A_n(2\sigma t)\quad \text{for $t>0$,}\nonumber
	\end{align}
for some constant $c$,
where the first inequality holds by equation \eqref{AT2}, the second one by the $\Delta_2$-property of $\widetilde A$,  the third one by \eqref{dis-A-tildeA}, and the equality by the very definition of the function $A_n$.
\\ Combining inequalities \eqref{3.5}--\eqref{3.7}  yields
\begin{align}
	\label{3.8}
	\intort F(|u|) (|u|- t) \,d\HN \le \eps c \intot A(|\nabla u|) \,dx+ \eps c \mu(t) A_n(\sigma t) \quad \text{for $t>t_1$,}
	\end{align}
for some constant $c$.

\color{black}
From equations \eqref{robin-1}--\eqref{3}, \eqref{3.4}, and \eqref{3.8} we infer that, if  $\eps$ is chosen sufficiently small, then
	\begin{align}
	\label{4}
	&\intot A(|\nabla u|) \,dx \\
	& \le c \Bigg\{  \mu(t) A_n (c t)+
  \int_0^{\mu(t)} \widetilde A \l(r^{1/n}   M^{-1} \l({1}/{r} \r) \r) dr
\nonumber	  \\
	  & \quad+  \int_{\mu(t)}^{\infty} \widetilde A \l(r ^{-1/{n'}} \mu(t)  M^{-1} \l(1/\mu(t) \r) \r) dr
+ \int_{\mu(t)^{-1/n'}}^{\infty}  \widetilde A \l(r  \int_0^{1/r} N^{-1}(1/\rho) \;d\rho \r)\frac{dr}{r^{n'+1}}\Bigg\},    \nonumber
%
%
	\end{align}
{for $t>t_1$, for some constant $c$. Note that here we have again exploited the fact that $\widetilde A\in \Delta_2$.}
\\
Let $\Psi$ be the function defined as in \eqref{N}, with $\lambda =1$, and let 
$\Xi :[0, \infty) \to [0, \infty)$ be the function  given by
$$\Xi(s)  = \frac{1}{s}\int_{s^{-1/n'}}^{\infty}  \widetilde A \l(\frac r\eps  \int_0^{1/r} N^{-1}(1/\rho) \;d\rho \r)\frac{dr}{r^{n'+1}} \quad \text{for $s>0$.}$$
Inequality \eqref{4} can be rewritten as
	\begin{equation*}
	\frac{1}{\mu(t)} \intot A(|\nabla u|) \,dx \le c\l( A_n(c t)+ \Psi(\mu(t))+ \Xi(\mu(t))\r) \quad \text{for $t>t_1$.}
	\end{equation*}
Hence, via inequality  \eqref{dir-4.2} with $\alpha =n$ and property \eqref{Ak} one obtains that
	\begin{equation}
	\label{4.3}
	\int_t^{\infty} \mu(t)^{1/n'} d\tau \le  \kappa_4 \mu(t) A^{-1} \l(A_n(c t)+ \Psi(\mu(t))+ \Xi(\mu(t))\r) \quad \text{for $t>t_1$,}
	\end{equation}
for some constant $c$.
If   $\eta$ denotes  the function defined as in \eqref{y}, then
equation \eqref{4.3} in turn reads 
    \begin{equation}
        \label{robin-ineq-y}
    \eta(t) \le c \l(-\eta'(t)\r)^{n'} A^{-1} \l( A_n (ct)+ \Psi \big((-\eta'(t))^{n'}\big)+ \Xi \big((-\eta'(t))^{n'} \big)\r) \quad \text{for $t>t_1$.}
    \end{equation}
Now set
	\begin{align*}
	U& = \l\{t \ge t_1 : \, \Psi(\mu(t)) \ge \max\{ A_n(ct), \Xi(\mu(t))\}\r\},\\ 
		V & = \l\{t \ge t_1 : \, \Xi(\mu(t)) \ge \max\{ A_n(ct), \Psi(\mu(t))\}\r\}.
	\end{align*}
We claim that
	\begin{equation}
	\label{claim1}
	|U|< \infty \quad \text{and} \quad |V|< \infty.
	\end{equation}
One can show that $|U|< \infty$ via inequalities analogous to  \eqref{dis-dir-M} and \eqref{meas-U}.
\\ In order to prove that $|V|< \infty$, from equation \eqref{robin-ineq-y} and the definition of the set $V$ one deduces that
\begin{equation*}
	\eta(t) \le c\, (-\eta'(t))^{n'} A^{-1} \l( \Xi \big((-\eta'(t))^{n'} \big)\r)  = c\, \Theta \big(((-\eta'(t))^{n'})\big) \quad \text{for $t>t_1$},
	\end{equation*}
where $\Theta$ denotes the function  defined  in \eqref{def-G}. An argument analogous to the one which yields $|U|< \infty$, but relying upon 
Lemma \ref{prop:G} instead of Lemma \ref{prop:M}, 
tells us that $|V|< \infty$. Claim \eqref{claim1} hence follows. 
\\ We have thereby shown that the set  $Z= [t_1, \infty) \setminus \l(U \cup V\r)$ is such that $|Z|= \infty$. Since inequality  \eqref{dir-4.4.1} continues to hold in the present framework, with $\alpha =n$ and  the set $W$ replaced by $Z$, the same argument as in the proof of Theorem \ref{thm:neu} leads to a contradiction. This completes the proof.
\end{proof}


\section*{Compliance with Ethical Standards} 

\smallskip
\par\noindent 
{\bf Funding}. This research was partly funded by:  
\\ (i) Research project of the Italian Ministry of Education, University and Research (MIUR), Prin 2017 ``Nonlinear differential problems via variational, topological and set-valued methods'', grant number 2017AYM8XW (G.Barletta).
\\ (ii) Research Project  of the Italian Ministry of Education, University and
Research (MIUR) Prin 2017 ``Direct and inverse problems for partial differential equations: theoretical aspects and applications'',
grant number 201758MTR2 (A.Cianchi);
\\ (iii) GNAMPA   of the Italian INdAM - National Institute of High Mathematics (grant number not available)  (G.Barletta, A.Cianchi, G.Marino);   
\\  (iv) DFG via grant GZ: MA 10100/1-1, grant number 496629752 (G.Marino).
\smallskip
\par\noindent
{\bf Conflict of Interest}. The authors declare that they have no conflict of interest.

\end{document}